\documentclass[11pt]{article}

\usepackage[a4paper,margin=1in]{geometry}
\usepackage[T1]{fontenc}
\usepackage[utf8]{inputenc}
\usepackage{lmodern}
\usepackage{amsmath,amssymb,amsthm,mathtools}
\usepackage{microtype}
\usepackage[hidelinks]{hyperref}
\hypersetup{
  pdftitle={A Product Principle for Harmonic Schwarz Lemmas: Boxes, Polydiscs, and Metric Geometry},
  pdfauthor={Miljan Knezevic and Miodrag Mateljevic},
  pdfkeywords={harmonic mappings, prescribed values, product domains, Euclidean boxes, polydiscs, operator norm, strip method, Kobayashi metric, Finsler metrics, bang-bang extremals}
}
\newcommand{\doi}[1]{\href{https://doi.org/#1}{doi:#1}}
\numberwithin{equation}{section}

\theoremstyle{plain}
\newtheorem{theorem}{Theorem}[section]
\newtheorem{lemma}[theorem]{Lemma}
\newtheorem{proposition}[theorem]{Proposition}
\newtheorem{corollary}[theorem]{Corollary}
\theoremstyle{definition}

\theoremstyle{remark}
\newtheorem{remark}[theorem]{Remark}

\newcommand{\D}{\mathbb D}
\newcommand{\T}{\mathbb T}
\newcommand{\C}{\mathbb C}
\newcommand{\R}{\mathbb R}
\newcommand{\dm}{\,dm}
\newcommand{\re}{\operatorname{Re}}

\newcommand{\Hyp}{\operatorname{Hyp}}
\newcommand{\sgn}{\operatorname{sgn}}

\title{A Product Principle for Harmonic Schwarz Lemmas: Boxes, Polydiscs, and Metric Geometry}
\author{
Miljan Kne\v{z}evi\'c%
\thanks{Faculty of Mathematics, University of Belgrade,
Studentski trg 16, 11000 Belgrade, Serbia.
E-mail: \texttt{miljan.knezevic@matf.bg.ac.rs}.}
\and
Miodrag Mateljevi\'c%
\thanks{Faculty of Mathematics, University of Belgrade,
Studentski trg 16, 11000 Belgrade, Serbia.
E-mail: \texttt{miodrag@matf.bg.ac.rs}.}
}
\date{}

\begin{document}

\maketitle

\begin{abstract}
We establish an exact product principle for the Euclidean operator norm of
differentials of harmonic maps. For a bounded domain \(G\subset\R^m\)
and \(p\in G\), let \(M_G(p)\) denote the supremum of \(\|dF_0\|\) over
harmonic maps \(F:\D\to G\) with \(F(0)=p\). For bounded domains
\(G_j\subset\R^{m_j}\), we prove
\[
 M_{G_1\times\cdots\times G_N}(p_1,\ldots,p_N)^2
 =\sum_{j=1}^N M_{G_j}(p_j)^2.
\]
The theorem separates the geometry of the individual factors from the
Euclidean geometry of the product: the factor extremal constants combine by
a sum-of-squares law, while equality is governed by a single compatibility
condition, namely a common maximizing direction for the component
differentials. Neither convexity nor attainment of the factor suprema is
required.

Combining the product principle with the sharp interval and disk factor
problems yields exact operator-norm estimates and all equality cases for
harmonic maps into boxes and polydiscs. In both families, for every extremal
map, the real differential at the origin has one-dimensional image.

The same factor constants also define coordinatewise metrics for which
the harmonic contraction estimate is sharp when the source disk is equipped
with its Poincar\'e metric of curvature \(-1\). For boxes, the resulting metric is complete and equals
twice the restriction of the Kobayashi-Royden metric of the product of
vertical strips. For polydiscs, the harmonic product metric is pointwise
maximal among contracting metrics of the form
\(\max_j a_j(p)|v_j|\), with \(a_j(p)>0\). It is strictly smaller than twice
the Kobayashi-Royden metric on every nonzero tangent vector, and its induced
path metric is incomplete.
\end{abstract}



\section{Introduction}\label{sec:introduction}

For a bounded domain \(G\subset\R^m\) and a point \(p\in G\), consider
\begin{equation}\label{eq:intro-MG}
 M_G(p)=
 \sup\{\|dF_0\|:F:\D\to G\ \text{harmonic},\ F(0)=p\},
\end{equation}
where the differential is viewed as a real linear map from \(\R^2\) to
\(\R^m\), and we use its Euclidean operator norm. All vector-valued harmonic
maps are understood componentwise. Harmonicity is preserved under conformal
reparametrization of the disk, but not under general nonlinear changes of the
target. The prescribed value therefore remains part of the extremal problem,
along with the target domain and the norm used for the differential.

The main structural result is an exact product theorem for the quantity in
\eqref{eq:intro-MG}. If \(G_j\subset\R^{m_j}\) are bounded domains and
\(G=G_1\times\cdots\times G_N\) is equipped with the Euclidean product norm,
then
\begin{equation}\label{eq:intro-product-principle}
 M_G(p_1,\ldots,p_N)^2
 =
 \sum_{j=1}^NM_{G_j}(p_j)^2.
\end{equation}
The identity, proved in Theorem~\ref{thm:product-principle}, requires neither
convexity nor attainment of the factor suprema. Independent rotations of the
parameter disk align maximizing directions of near-extremal factor maps and
give the reverse inequality. The equality problem is equally rigid:
Theorem~\ref{thm:product-equality} shows that a product map is extremal
exactly when every component is extremal and all component differentials
attain their norms on a common unit vector in \(\R^2\).

The essential point of the product theorem is this separation of two kinds of
geometry. All target-specific information is contained in the factor
quantities \(M_{G_j}(p_j)\); once they are known, the sharp constant on the
product requires no further target-specific optimization. The product itself
contributes the Euclidean sum-of-squares law and, in the equality case, the
common maximizing direction in the parameter plane. This distinction allows
the same theorem to treat factors with very different extremal boundary
behavior and very different intrinsic metrics.

For interval factors, the principle gives an explicit model immediately.
Consider
\[
 Q=\prod_{j=1}^n(a_j,b_j)\subset\R^n.
\]
For \(p=(p_1,\ldots,p_n)\in Q\), set
\begin{equation}\label{eq:intro-Aj}
 A_j(p_j)=\frac{2(b_j-a_j)}{\pi}
 \sin\frac{\pi(p_j-a_j)}{b_j-a_j}.
\end{equation}
The exact prescribed-value formula is
\begin{equation}\label{eq:intro-box-formula}
 M_Q(p)=
 \left(\sum_{j=1}^nA_j(p_j)^2\right)^{1/2}.
\end{equation}
Each factor constant is the sharp interval constant. After normalization to
\((0,1)\), the boundary function of a scalar extremal is the characteristic
function of an arc of prescribed length. Equality in the box estimate requires
these scalar extremals to have collinear gradients. Consequently, for every
box extremal, the real differential at the origin has one-dimensional image.
For rectangles, Lemma~\ref{lem:operator-norm} turns
\eqref{eq:intro-box-formula} into the exact formula for
\[
 |f_z(0)|+|f_{\overline z}(0)|.
\]

The factor constants also lead to two distinct infinitesimal geometries. The
Euclidean extremal density reflects the \(\ell^2\) combination in
\eqref{eq:intro-product-principle}, whereas coordinatewise contraction
produces an \(\ell^\infty\) product metric. For boxes, the latter has a
classical complex-geometric interpretation. Replacing \((a_j,b_j)\) by the vertical
strip
\[
 S_j=\{w\in\C:a_j<\re w<b_j\}
\]
gives the complexification
\[
 S_Q=\prod_{j=1}^nS_j,\qquad Q=S_Q\cap\R^n.
\]
The interval \((a_j,b_j)\) is a hyperbolic geodesic in \(S_j\), and the sharp
scalar density \(2/A_j\) is exactly the restriction of the curvature \(-1\)
Poincar\'e density of the strip. Consequently, the product metric on \(Q\) is
twice the restriction of the Kobayashi-Royden metric of \(S_Q\), and its path
metric is twice the restricted Kobayashi distance and is complete. The strip
method therefore recovers the sharp scalar estimate and, at the same time,
identifies its density with an intrinsic complex-geometric metric.

Every strip \(S_j\) is biholomorphic to \(\D\), so \(S_Q\) is biholomorphic
to the polydisc \(\D^n\). This equivalence concerns Kobayashi geometry, but it
does not transport the harmonic mapping problem: postcomposition by a
nonlinear target biholomorphism need not preserve harmonicity. The disk-valued
factor must therefore be analyzed directly.

For \(0\le r<1\), define
\begin{equation}\label{eq:intro-Mdisk}
 \mathcal M(r)=
 \sup\{\|df_0\|:f:\D\to\D\ \text{harmonic},\ |f(0)|=r\}.
\end{equation}
We use the disk extremal theorem from
\cite{KnezevicMateljevic2026}, recalled with proof in
Section~\ref{sec:disk-factor}. No conformality, injectivity, or orientation
assumption is imposed. At the center, \(\mathcal M(0)=4/\pi\). For
\(0<r<1\), there is a unique \(\lambda(r)>0\) such that
\begin{equation}\label{eq:intro-lambda}
 r=
 \int_0^{2\pi}
 \frac{\lambda(r)}
 {\sqrt{\lambda(r)^2+\cos^2t}}\,
 \frac{dt}{2\pi},
\end{equation}
and the exact constant is
\begin{equation}\label{eq:intro-Mdisk-formula}
 \mathcal M(r)=
 2\int_0^{2\pi}
 \frac{\cos^2t}
 {\sqrt{\lambda(r)^2+\cos^2t}}\,
 \frac{dt}{2\pi}.
\end{equation}
A supporting inequality and a comparison of derivative directions determine
the optimal direction and the equality cases. If \(p=re^{i\beta}\ne0\), the
extremal boundary functions are
\begin{equation}\label{eq:intro-disk-extremal}
 \Phi_{p,\alpha}(e^{it})
 =
 e^{i\beta}
 \frac{\lambda(r)+i\cos(t-\alpha)}
 {\sqrt{\lambda(r)^2+\cos^2(t-\alpha)}}.
\end{equation}
Their differentials attain the operator norm on \(\R e^{i\alpha}\), vanish on
the orthogonal line, and have image perpendicular to the prescribed value
\(p\). Proposition~\ref{prop:disk-equality} shows that these are all
extremals for \(r>0\); the centered case is described separately there.

Applying the product theorem to disk factors gives, for
\(p=(p_1,\ldots,p_n)\in\D^n\),
\begin{equation}\label{eq:intro-polydisc}
 \sup\{\|dF_0\|:F:\D\to\D^n\ \text{harmonic},\ F(0)=p\}
 =
 \left(
 \sum_{j=1}^n\mathcal M(|p_j|)^2
 \right)^{1/2}.
\end{equation}
In particular, the centered constant is \(4\sqrt n/\pi\). All equality cases
are obtained by combining the scalar disk extremals with the common-direction
condition. For every extremal map into a box or a polydisc, the real
differential at the origin has one-dimensional image.

The exact disk constant also defines the density
\[
 \rho_h(w)=\frac{2}{\mathcal M(|w|)}.
\]
Its coordinatewise product metric contracts every harmonic map
\(\D\to\D^n\) with sharp infinitesimal constant. Among coordinatewise
norms of the form \(\max_j a_j(p)|v_j|\), with \(a_j(p)>0\), it is
pointwise maximal. Its geometry
is nevertheless different from the strip product: it is strictly smaller than
twice the Kobayashi-Royden metric on every nonzero tangent vector and has
finite distance to the Euclidean boundary. Thus the strip product metric is
complete, whereas the sharp harmonic product metric on the polydisc is not.

The box and polydisc applications place the product principle in two very
different settings. Box extremals come from endpoint-valued interval
boundary functions and lead to a complete strip geometry. Polydisc extremals
come from unimodular disk boundary functions and lead to a sharp harmonic
metric that is incomplete. What is common to the two theories is therefore
not the geometry of the factors, but the product mechanism that combines
their sharp constants and equality directions.

The factor problems themselves belong to different parts of harmonic Schwarz
theory. Chen proved the sharp cosine estimate and its equality cases for real
harmonic maps into \((-1,1)\) \cite{Chen2013}. The direct boundary proof given
here is retained because its arc description is precisely the form needed for
box extremals and for the product equality analysis. Classical disk estimates
of Heinz and Colonna give the sharp centered Schwarz bound and uniform
differential estimates \cite{Heinz,Colonna}; see also
\cite{DurenBook,KalajVuorinen}. Duren and Schober developed a variational
theory for univalent harmonic maps onto convex regions and related linear
extremal problems for harmonic mappings of the disk
\cite{DurenSchoberVariational,DurenSchober}. Further extremal problems for
harmonic mappings into convex regions were studied by Wegmann \cite{Wegmann},
and the conformal-at-a-point problem for harmonic maps into real balls is
treated by Forstneri\v{c} and Kalaj \cite{ForstnericKalaj}. The unrestricted
prescribed-value disk factor used below was determined in
\cite{KnezevicMateljevic2026}.

The new structural step in the present paper is the exact passage from sharp
factor problems to finite Euclidean products, together with the corresponding
equality criterion. It is independent of the particular geometry of the
factors. This is also distinct from the prescribed-value problem studied in
\cite{KnezevicMateljevicTarget2026}, where one prescribes the image plane of
the differential at the origin and requires the differential to be conformal
there. Here the Euclidean operator norm is unrestricted, and the product
theorem transfers sharp factor constants and equality cases to product
targets, where the same constants also generate the coordinatewise contraction
metrics studied below.

Related several-variable work concerns Schwarz-Pick estimates and boundary
behavior for pluriharmonic mappings on polydiscs or between higher-dimensional
targets \cite{ChenRasila,ChenHamada,HuangZhaoLi}. The setting here is
different: the parameter domain is the one-dimensional disk, the target is
\(\D^n\), the target value is prescribed, and the optimized quantity is the
Euclidean operator norm of the full real differential. Harmonic densities and
the associated distance-decreasing viewpoint were considered earlier in
\cite{Mateljevic2018,ArsenovicGajicMateljevic}; the density \(\rho_h\) above
is determined here by the exact prescribed-value constant \(\mathcal M\).
Strip methods and their hyperbolic interpretation are developed in
\cite{Mateljevic2018,MateljevicSvetlik,KMS2025}.

The extremal boundary functions admit a complementary description through
duality and optimal control. For interval controls, pointwise maximization of
a linear Hamiltonian selects the endpoints and produces bang-bang boundary
values. For nonzero prescribed disk values, the corresponding support problem
selects a unique point of the unit circle for almost every boundary parameter
and produces the smooth unimodular boundary function
\eqref{eq:intro-disk-extremal}. In products, these selections occur
factorwise; the Euclidean operator norm adds the common maximizing direction
in \(\R^2\).

The paper is organized as follows. Section~\ref{sec:preliminaries} records the
Poisson formulas and operator-norm facts used throughout.
Section~\ref{sec:interval} treats the interval factor. Sections~\ref{sec:rectangle}
through \ref{sec:boxes} develop the rectangle and box models, including all
equality cases, and Section~\ref{sec:product-principle} isolates the general
product principle behind them. Section~\ref{sec:pointwise-general} gives its
pointwise form and the extension to complete conformal metrics on the domain.
Section~\ref{sec:strip-product} develops the strip complexification and its
Kobayashi interpretation. Section~\ref{sec:disk-factor} gives a self-contained
treatment of the disk-valued factor, and Section~\ref{sec:polydisc} derives
the polydisc theorem and the sharp harmonic product metric.
Section~\ref{sec:duality-control} gives the dual and control interpretations,
and Section~\ref{sec:conclusion} summarizes the resulting product geometry.

\section{Preliminaries}\label{sec:preliminaries}

We write
\[
 \D=\{z\in\C:|z|<1\},\qquad \T=\partial\D,
\]
and denote by \(m\) normalized Lebesgue measure on \(\T\). Thus
\[
 \int_\T\psi\dm=\frac1{2\pi}\int_0^{2\pi}\psi(e^{it})\,dt.
\]
For \(U\in L^\infty(\T)\), its Poisson integral is
\[
 P[U](re^{i\theta})
 =\int_\T P_r(\theta-t)U(e^{it})\dm,
 \qquad
 P_r(s)=\frac{1-r^2}{1-2r\cos s+r^2}.
\]

The following lemma records the boundary representation and the first
derivative formulas.

\begin{lemma}\label{lem:real-poisson}
Let \(u\) be a bounded real-valued harmonic function on \(\D\). Then there is a
unique \(U\in L^\infty(\T)\) such that \(u=P[U]\). Moreover,
\begin{equation}\label{eq:real-poisson-data}
 u(0)=\int_\T U\dm,
 \qquad
 u_x(0)=2\int_\T U(e^{it})\cos t\dm,
 \qquad
 u_y(0)=2\int_\T U(e^{it})\sin t\dm.
\end{equation}
Consequently,
\begin{equation}\label{eq:real-poisson-grad}
 |\nabla u(0)|
 =2\sup_{\theta\in\R}
 \int_\T U(e^{it})\cos(t-\theta)\dm.
\end{equation}
If \(u(\D)\subset(a,b)\), then \(a\leq U\leq b\) almost everywhere.
Conversely, if \(a\leq U\leq b\) almost everywhere and \(U\) is not almost
everywhere equal to either constant \(a\) or \(b\), then
\(P[U](\D)\subset(a,b)\).
\end{lemma}

\begin{proof}
The Poisson representation and the almost-everywhere radial limits are
standard for bounded harmonic functions; see \cite{AxlerBourdonRamey}.  The
bounds on \(U\) follow from the radial limits.  The expansion
\[
 P_r(\theta-t)=1+2r\cos(\theta-t)+O(r^2)
\]
is uniform in \(t\) as \(r\to0\) and gives
\eqref{eq:real-poisson-data} by differentiation.  Formula
\eqref{eq:real-poisson-grad} is the dual characterization of the Euclidean
norm of the gradient; replacing \(\theta\) by \(\theta+\pi\) changes the
sign.  Finally, the Poisson kernel is strictly positive.  If \(P[U](z)=a\)
at an interior point, then \(U=a\) almost everywhere, and the same argument
applies to \(b\).  This proves the last assertion.
\end{proof}

The same formulas apply to complex-valued boundary functions.

\begin{lemma}\label{lem:complex-poisson}
Let \(f=P[\Phi]\), where \(\Phi\in L^\infty(\T,\C)\). Then
\[
 f(0)=\int_\T\Phi\dm
\]
and, for every \(\theta\in\R\),
\begin{equation}\label{eq:complex-directional}
 df_0(e^{i\theta})
 =
 2\int_\T\Phi(e^{it})\cos(t-\theta)\dm.
\end{equation}
Consequently,
\begin{equation}\label{eq:complex-poisson-norm}
 \|df_0\|
 =
 2\sup_{\theta\in\R}
 \left|
 \int_\T\Phi(e^{it})\cos(t-\theta)\dm
 \right|.
\end{equation}
\end{lemma}

\begin{proof}
The assertion follows by applying Lemma~\ref{lem:real-poisson} to the real
and imaginary parts of \(f\).
\end{proof}

For a \(C^1\) complex-valued map \(f\), we use
\[
 f_z=\frac12(f_x-if_y),\qquad
 f_{\overline z}=\frac12(f_x+if_y).
\]
The operator norm is given by the following elementary identity.

\begin{lemma}\label{lem:operator-norm}
For every \(C^1\) complex-valued map \(f\) and every point \(z\) in its
domain,
\begin{equation}\label{eq:norm-identity}
 \|df_z\|=|f_z(z)|+|f_{\overline z}(z)|,
\end{equation}
where \(df_z:\R^2\to\R^2\) is the real differential.
\end{lemma}

\begin{proof}
For a unit vector \(e^{i\alpha}\),
\[
 df_z(e^{i\alpha})
 =
 f_z(z)e^{i\alpha}+f_{\overline z}(z)e^{-i\alpha}.
\]
The triangle inequality gives the upper estimate. There is a choice of
\(\alpha\) for which the two nonzero terms have the same argument; if one
term vanishes, equality is immediate.
\end{proof}

The product arguments require a block version of the elementary comparison
between the operator norm and the Hilbert-Schmidt norm.

\begin{lemma}\label{lem:block-operator}
Let \(E_1,\ldots,E_N\) be finite-dimensional Euclidean spaces and let
\(A_j:\R^2\to E_j\) be linear. For
\[
 A=(A_1,\ldots,A_N):
 \R^2\longrightarrow E_1\oplus\cdots\oplus E_N
\]
with the Euclidean product norm,
\begin{equation}\label{eq:block-operator}
 \|A\|^2\leq\sum_{j=1}^N\|A_j\|^2.
\end{equation}
Equality holds if and only if there is a unit vector \(v\in\R^2\) such that
\begin{equation}\label{eq:common-direction}
 |A_jv|=\|A_j\|,
 \qquad j=1,\ldots,N.
\end{equation}
\end{lemma}

\begin{proof}
For every unit vector \(v\),
\[
 |Av|^2=\sum_{j=1}^N|A_jv|^2
 \leq\sum_{j=1}^N\|A_j\|^2,
\]
which proves \eqref{eq:block-operator}.  If equality holds, a unit vector
\(v\) attaining \(\|A\|\) satisfies
\[
 \sum_{j=1}^N|A_jv|^2
 =
 \sum_{j=1}^N\|A_j\|^2.
\]
Every summand on the left is bounded above by the corresponding summand on
the right, so equality holds term by term and
\eqref{eq:common-direction} follows.  The converse is immediate.
\end{proof}

For scalar blocks, the common-direction condition becomes the familiar
collinearity condition.

\begin{lemma}\label{lem:operator-HS}
Let \(A:\R^2\to\R^n\) be linear, and let
\(r_1,\ldots,r_n\in\R^2\) be the rows of its matrix in orthonormal
coordinates. Then
\begin{equation}\label{eq:operator-HS}
 \|A\|^2\leq\sum_{j=1}^n|r_j|^2.
\end{equation}
Equality holds if and only if the image of \(A\) has dimension at most one,
equivalently, if and only if the nonzero row vectors are collinear.
\end{lemma}

\begin{proof}
This is Lemma~\ref{lem:block-operator} with \(E_j=\R\).  A unit vector
maximizes the linear functional determined by a nonzero row precisely when
it is parallel or antiparallel to that row.  Hence a common maximizing line
exists exactly when the nonzero rows are collinear.
\end{proof}
\section{The interval-valued extremal problem}\label{sec:interval}

The scalar interval problem provides the coordinate estimates for rectangles
and boxes. Among boundary functions with prescribed mean and values in an
interval, the largest value of
\(\int_\T U(e^{it})\cos t\dm\) is obtained by placing the upper endpoint
where \(\cos t\) is largest. The same comparison also determines the
maximizing boundary function uniquely.

\begin{lemma}\label{lem:rearrangement}
Let \(0<q<1\). Then
\begin{equation}\label{eq:rearrangement}
 \sup\left\{\int_\T U(e^{it})\cos t\dm:
 U\in L^\infty(\T),\ 0\leq U\leq1\ \text{a.e.},
 \int_\T U\dm=q\right\}=\frac{\sin(\pi q)}{\pi}.
\end{equation}
The supremum is attained by
\[
 U_q(e^{it})=\chi_{[-\pi q,\pi q]}(t),
 \qquad -\pi<t\leq\pi.
\]
Equality holds if and only if \(U=U_q\) almost everywhere. After replacing \(\cos t\) by \(\cos(t-\theta)\), the maximizing function is the characteristic function of the corresponding rotated arc.
\end{lemma}

\begin{proof}
We set \(E_q=\{e^{it}: |t|\leq \pi q\}\) and \(\gamma=\cos(\pi q)\). Then \(m(E_q)=q\) and
\[
 \int_{E_q}\cos t\dm=\frac{1}{2\pi}\int_{-\pi q}^{\pi q}\cos t\,dt
 =\frac{\sin(\pi q)}{\pi}.
\]
For every \(U\) satisfying the constraints in
\eqref{eq:rearrangement}, we have
\(\int_\T(U-\chi_{E_q})\dm=0\), and therefore
\[
 \int_\T (U-\chi_{E_q})\cos t\dm
 =\int_\T (U-\chi_{E_q})(\cos t-\gamma)\dm.
\]
On \(E_q\), we have \(U-1\leq0\) and \(\cos t-\gamma\geq0\). On
\(\T\setminus E_q\), we have \(U\geq0\) and \(\cos t-\gamma\leq0\). Hence
the last integral is non-positive, proving \eqref{eq:rearrangement}. Since
\(\cos t-\gamma\) has a strict sign away from the two endpoints of the arc,
equality forces \(U=1\) almost everywhere on \(E_q\) and \(U=0\) almost
everywhere on \(\T\setminus E_q\).
\end{proof}

The estimate in the following theorem is Chen's sharp Schwarz lemma for real
harmonic functions \cite{Chen2013}, written for an arbitrary interval. The
boundary proof below gives the equality characterization in the form needed
for the product constructions. The strip proof is recalled in
Section~\ref{sec:strip-product}.

\begin{theorem}\label{thm:interval}
Let \(a<b\), let \(x\in(a,b)\), and let \(u:\D\to(a,b)\) be harmonic with \(u(0)=x\). Then
\begin{equation}\label{eq:interval-estimate}
 |\nabla u(0)|\leq
 \frac{2(b-a)}{\pi}\sin\frac{\pi(x-a)}{b-a}.
\end{equation}
The estimate is sharp for every \(x\in(a,b)\). Equivalently,
\begin{equation}\label{eq:interval-cos-form}
 |\nabla u(0)|\leq
 \frac{2(b-a)}{\pi}
 \cos\frac{\pi(2x-a-b)}{2(b-a)}.
\end{equation}
If \(U\) denotes the radial boundary function of \(u\) and
\(q=(x-a)/(b-a)\), then equality holds in
\eqref{eq:interval-estimate} if and only if, up to a set of measure zero,
\[
 \frac{U-a}{b-a}=\chi_I,
\]
where \(I\subset\T\) is an arc of angular length \(2\pi q\).
\end{theorem}

\begin{proof}
We set \(q=(x-a)/(b-a)\in(0,1)\) and write \(u=P[U]\), where
\(a\leq U\leq b\) a.e. and \(\int_\T U\dm=x\). We then set
\(V=(U-a)/(b-a)\). Thus \(0\leq V\leq1\) and
\(\int_\T V\dm=q\). By Lemma~\ref{lem:real-poisson}, for each direction
\(\theta\),
\[
 D_\theta u(0)=2(b-a)\int_\T V(e^{it})\cos(t-\theta)\dm.
\]
By applying Lemma~\ref{lem:rearrangement} after rotation, we obtain 
\[
 D_\theta u(0)\leq \frac{2(b-a)}{\pi}\sin(\pi q).
\]
The same estimate in the opposite direction \(\theta+\pi\) gives
the corresponding lower bound. Therefore
\[
 |D_\theta u(0)|\leq \frac{2(b-a)}{\pi}\sin(\pi q).
\]
Taking the supremum over \(\theta\) proves \eqref{eq:interval-estimate}. The form \eqref{eq:interval-cos-form} follows from
\[
 \sin\frac{\pi(x-a)}{b-a}
 =\cos\frac{\pi(2x-a-b)}{2(b-a)}.
\]
Sharpness follows by taking
\[
 U_x(e^{it})=a+(b-a)\chi_{[-\pi q,\pi q]}(t)
\]
and setting \(u_x=P[U_x]\). Since \(0<q<1\), the boundary function is not
essentially constant, so the maximum principle gives
\(u_x(\D)\subset(a,b)\). Lemmas~\ref{lem:rearrangement}
and~\ref{lem:real-poisson} give equality in
\eqref{eq:interval-estimate}.

Conversely, if equality holds in
\eqref{eq:interval-estimate}, then, since \(x\in(a,b)\), the right-hand side
of \eqref{eq:interval-estimate} is positive and
\(\nabla u(0)\neq0\). We choose \(\theta\) in the direction of
\(\nabla u(0)\), so that
\(D_\theta u(0)=|\nabla u(0)|\). Equality must then hold in
Lemma~\ref{lem:rearrangement}, applied after rotation. Consequently, up to a
null set, the normalized boundary function is the characteristic function of
an arc of angular length \(2\pi q\). Conversely, every such boundary function
gives equality. This completes the characterization.
\end{proof}

For the symmetric interval \((-1,1)\), the estimate takes the following
particularly simple form.

\begin{remark}\label{rem:symmetric-interval}
Theorem~\ref{thm:interval} gives the sharp estimate
\[
 |\nabla u(0)|\leq \frac4\pi\cos\frac{\pi u(0)}2,
 \qquad u:\D\to(-1,1)\text{ harmonic}.
\]
This is precisely the form obtained from the strip metric in
Section~\ref{sec:strip-product}.
\end{remark}

\section{Rectangle-valued harmonic mappings}\label{sec:rectangle}

For a rectangle, the coordinate constraints separate, but the operator norm
of the full differential does not: it is not the sum of the two component
gradient bounds. The geometry of the two rows of the differential is the only
additional ingredient.

For the rectangle
\[
 R=(a,b)\times(c,d),\qquad a<b,
 \qquad c<d,
\]
and a point \(p=p_1+ip_2\in R\), we define
\begin{equation}\label{eq:AR-def}
 A_R(p_1)=\frac{2(b-a)}{\pi}\sin\frac{\pi(p_1-a)}{b-a},
\end{equation}
\begin{equation}\label{eq:BR-def}
 B_R(p_2)=\frac{2(d-c)}{\pi}\sin\frac{\pi(p_2-c)}{d-c},
\end{equation}
and
\begin{equation}\label{eq:MR-def}
 M_R(p)=\bigl(A_R(p_1)^2+B_R(p_2)^2\bigr)^{1/2}.
\end{equation}

The exact rectangle estimate is as follows.

\begin{theorem}\label{thm:rectangle}
Let \(R=(a,b)\times(c,d)\) be an open rectangle, let \(p=p_1+ip_2\in R\), and let \(f:\D\to R\) be harmonic with \(f(0)=p\). Then
\begin{equation}\label{eq:rectangle-main-bound}
 |f_z(0)|+|f_{\overline z}(0)|\leq M_R(p).
\end{equation}
The estimate is sharp for every \(p\in R\). Equivalently,
\begin{equation}\label{eq:rectangle-sup}
 \sup\left\{
 |f_z(0)|+|f_{\overline z}(0)|:
 f:\D\to R\ \text{harmonic},\ f(0)=p
 \right\}=M_R(p).
\end{equation}
\end{theorem}

\begin{proof}
We write \(f=u+iv\). Then \(u:\D\to(a,b)\), \(v:\D\to(c,d)\), \(u(0)=p_1\), and \(v(0)=p_2\). By Theorem~\ref{thm:interval},
\begin{equation}\label{eq:component-u}
 |\nabla u(0)|\leq A_R(p_1),
\end{equation}
and
\begin{equation}\label{eq:component-v}
 |\nabla v(0)|\leq B_R(p_2).
\end{equation}
The matrix of the real differential \(df_0:\R^2\to\R^2\) has rows
\(\nabla u(0)\) and \(\nabla v(0)\). Lemma~\ref{lem:operator-HS} gives
\[
 \|df_0\|^2\leq |\nabla u(0)|^2+|\nabla v(0)|^2.
\]
Using \eqref{eq:component-u} and \eqref{eq:component-v},
\[
 \|df_0\|\leq
 \bigl(A_R(p_1)^2+B_R(p_2)^2\bigr)^{1/2}=M_R(p).
\]
By Lemma~\ref{lem:operator-norm}, this is \eqref{eq:rectangle-main-bound}.

Sharpness follows from an explicit extremal construction. We set
\[
 q_1=\frac{p_1-a}{b-a},
 \qquad
 q_2=\frac{p_2-c}{d-c}.
\]
We define the boundary functions on \((-\pi,\pi]\) by
\[
 U(e^{it})=a+(b-a)\chi_{[-\pi q_1,\pi q_1]}(t),
\]
and
\[
 V(e^{it})=c+(d-c)\chi_{[-\pi q_2,\pi q_2]}(t).
\]
We set \(u=P[U]\), \(v=P[V]\), and \(f=u+iv\). Then \(f(\D)\subset R\), \(f(0)=p\), and Theorem~\ref{thm:interval} gives
\[
 |\nabla u(0)|=A_R(p_1),
 \qquad
 |\nabla v(0)|=B_R(p_2).
\]
The two gradients are nonzero because \(q_1,q_2\in(0,1)\), and they are
parallel since the two arcs are centered at the same point of \(\T\).
Hence Lemma~\ref{lem:operator-HS} shows that the operator norm equals the
Hilbert-Schmidt norm. Therefore
\[
 \|df_0\|=\bigl(|\nabla u(0)|^2+|\nabla v(0)|^2\bigr)^{1/2}=M_R(p).
\]
Lemma~\ref{lem:operator-norm} gives equality in \eqref{eq:rectangle-main-bound}.
\end{proof}

The formula is particularly simple for the unit square.

\begin{corollary}\label{cor:unit-square}
Let \(Q=(0,1)^2\), and let \(p=p_1+ip_2\in Q\). Then
\[
 \sup\left\{
 |f_z(0)|+|f_{\overline z}(0)|:
 f:\D\to Q\ \text{harmonic},\ f(0)=p
 \right\}
 =\frac2\pi\bigl(\sin^2(\pi p_1)+\sin^2(\pi p_2)\bigr)^{1/2}.
\]
\end{corollary}

\begin{proof}
This is Theorem~\ref{thm:rectangle} with \(a=c=0\) and \(b=d=1\).
\end{proof}

At the center of the rectangle, both scalar factors attain their largest
values.

\begin{corollary}\label{cor:center-rect}
Let \(p_0=(a+b)/2+i(c+d)/2\) be the center of \(R=(a,b)\times(c,d)\). Then
\[
 \sup\left\{
 |f_z(0)|+|f_{\overline z}(0)|:
 f:\D\to R\ \text{harmonic},\ f(0)=p_0
 \right\}
 =\frac2\pi\bigl((b-a)^2+(d-c)^2\bigr)^{1/2}.
\]
In particular, for the center of the unit square the exact value is
\(2\sqrt2/\pi\).
\end{corollary}

\begin{proof}
At the center, both sine factors in \eqref{eq:AR-def} and \eqref{eq:BR-def} are equal to \(1\). The statement follows from Theorem~\ref{thm:rectangle}.
\end{proof}

The formula also gives the exact limiting behavior as the prescribed point
approaches the boundary.

\begin{corollary}\label{cor:boundary-rect}
Let \(p=p_1+ip_2\in R\). If \(p_1\to a^+\) or \(p_1\to b^-\), while \(p_2\) remains fixed in \((c,d)\), then
\[
 M_R(p)\to B_R(p_2).
\]
Similarly, if \(p_2\to c^+\) or \(p_2\to d^-\), while \(p_1\) remains fixed in \((a,b)\), then
\[
 M_R(p)\to A_R(p_1).
\]
In particular, \(M_R(p)\to0\) as \(p\) tends to a vertex of \(R\).
\end{corollary}

\begin{proof}
The functions \(A_R\) and \(B_R\) extend continuously to the closed intervals, with \(A_R(a)=A_R(b)=0\) and \(B_R(c)=B_R(d)=0\). The assertions follow from \eqref{eq:MR-def}.
\end{proof}

\section{Extremal boundary functions for rectangles}\label{sec:rectangle-extremals}

The proof of Theorem~\ref{thm:rectangle} constructs extremal maps.
We now characterize all maps that attain equality by combining the interval
equality cases with the collinearity condition for the component gradients.

We set \(q_1=(p_1-a)/(b-a)\) and \(q_2=(p_2-c)/(d-c)\). For \(q\in(0,1)\) and \(\theta\in\R\), we use
\[
 I_{q,\theta}
 =\left\{e^{it}\in\T:
 \left|\operatorname{Arg}\bigl(e^{i(t-\theta)}\bigr)\right|\leq\pi q
 \right\},
\]
where \(\operatorname{Arg}\) takes values in \((-\pi,\pi]\). Thus
\(I_{q,\theta}\) is the closed arc of angular length \(2\pi q\) centered at
\(e^{i\theta}\). For each \(\theta\), we define
\begin{equation}\label{eq:rect-ext-boundary}
 \Phi_{p,\theta}(e^{it})=
 a+(b-a)\chi_{I_{q_1,\theta}}(e^{it})
 +i\bigl(c+(d-c)\chi_{I_{q_2,\theta}}(e^{it})\bigr).
\end{equation}
The Poisson integral \(P[\Phi_{p,\theta}]\) maps \(\D\) into \(R\), has
value \(p\) at the origin, and attains equality in
Theorem~\ref{thm:rectangle}. Using the same center \(\theta\) makes the two
component gradients point in the same direction; choosing opposite centers
makes them point in opposite directions along the same line.

\begin{proposition}\label{prop:rect-equality}
Suppose \(f=u+iv:\D\to R\) is harmonic, \(f(0)=p=p_1+ip_2\), and \(f\) attains equality in Theorem~\ref{thm:rectangle}. If \(U\) and \(V\) are the radial boundary values of \(u\) and \(v\), then, up to null sets, the normalized boundary functions
\[
 \frac{U-a}{b-a},\qquad \frac{V-c}{d-c}
\]
are characteristic functions of arcs of lengths \(2\pi q_1\) and
\(2\pi q_2\), respectively. Moreover, the centers of the two arcs determine
the same unoriented line through the origin in the parameter plane \(\R^2\);
equivalently, their directions differ by \(0\) or \(\pi\) modulo \(2\pi\).
Conversely, every such choice gives an extremal mapping.
\end{proposition}

\begin{proof}
Equality in Theorem~\ref{thm:rectangle} means that equality holds in the chain
\[
 \|df_0\|^2\leq |\nabla u(0)|^2+|\nabla v(0)|^2
 \leq A_R(p_1)^2+B_R(p_2)^2.
\]
Thus equality holds in both component estimates, since the two deficits are
nonnegative and have sum zero. It also holds in the comparison of the
operator norm with the Hilbert-Schmidt norm. Since \(p\in R\), both best
component constants are positive. Lemma~\ref{lem:operator-HS} now shows that
the row vectors \(\nabla u(0)\) and \(\nabla v(0)\) are linearly dependent.

By the equality statement in Theorem~\ref{thm:interval}, each normalized
boundary function must be the characteristic function of an arc of the
prescribed length. Linear dependence of \(\nabla u(0)\) and
\(\nabla v(0)\) says exactly that the corresponding directions determine the
same unoriented line through the origin. The converse is the sharpness
construction in the proof of Theorem~\ref{thm:rectangle}.
\end{proof}

The boundary description also shows which points of the rectangle occur
almost everywhere.

\begin{remark}\label{rem:rectangle-boundary-values}
The extremal boundary functions in \eqref{eq:rect-ext-boundary} take values in the vertex set of the rectangle almost everywhere. Each real component takes only the two endpoint values of its interval, although not every vertex need occur. This contrasts with the disk, where a nonzero real linear functional is maximized at a unique point of the smooth boundary.
\end{remark}

\section{Boxes and hypercubes}\label{sec:boxes}

The rectangle argument extends directly to harmonic maps with values in a box
in \(\R^n\). The resulting square-root formula is a consequence of the
Euclidean operator norm and is not specific to dimension two.

For \(a_j<b_j\), \(j=1,\dots,n\), we set
\[
 Q=\prod_{j=1}^n(a_j,b_j)\subset\R^n.
\]
For \(p=(p_1,\dots,p_n)\in Q\), we define
\begin{equation}\label{eq:Aj-def}
 A_j(p_j)=\frac{2(b_j-a_j)}{\pi}
 \sin\frac{\pi(p_j-a_j)}{b_j-a_j},
 \qquad j=1,\dots,n,
\end{equation}
and
\begin{equation}\label{eq:MQ-def}
 M_Q(p)=\left(\sum_{j=1}^n A_j(p_j)^2\right)^{1/2}.
\end{equation}
If \(F=(f_1,\dots,f_n):\D\to Q\) is harmonic componentwise, \(\|dF_0\|\) denotes the operator norm of \(dF_0:\R^2\to\R^n\).

The same argument gives the exact estimate for an arbitrary box.

\begin{theorem}\label{thm:box}
Let \(Q=\prod_{j=1}^n(a_j,b_j)\subset\R^n\) be an open box, let \(p\in Q\), and let \(F:\D\to Q\) be harmonic with \(F(0)=p\). Then
\begin{equation}\label{eq:box-bound}
 \|dF_0\|\leq M_Q(p).
\end{equation}
The estimate is sharp for every \(p\in Q\).
\end{theorem}

\begin{proof}
Each component \(f_j:\D\to(a_j,b_j)\) is harmonic and satisfies \(f_j(0)=p_j\). By Theorem~\ref{thm:interval},
\[
 |\nabla f_j(0)|\leq A_j(p_j),
 \qquad j=1,\dots,n.
\]
The rows of the matrix of \(dF_0\) are the vectors \(\nabla f_j(0)\).
Lemma~\ref{lem:operator-HS} therefore gives
\[
 \|dF_0\|^2\leq\sum_{j=1}^n |\nabla f_j(0)|^2
 \leq\sum_{j=1}^n A_j(p_j)^2.
\]
This proves \eqref{eq:box-bound}.

To obtain equality, we choose extremal interval functions for all components with
arcs centered at the same point of \(\T\). Then all gradients
\(\nabla f_j(0)\) are parallel and have lengths \(A_j(p_j)\).
Lemma~\ref{lem:operator-HS} shows that the operator norm is exactly
\(M_Q(p)\).
\end{proof}

The equality cases are governed by the same collinearity condition as in the
rectangle case.

\begin{proposition}\label{prop:box-equality}
Suppose \(F=(f_1,\dots,f_n):\D\to Q\) is harmonic, \(F(0)=p\), and equality holds in \eqref{eq:box-bound}. If \(U_j\) are the radial boundary values of \(f_j\) and \(q_j=(p_j-a_j)/(b_j-a_j)\), then, up to null sets,
\[
 \frac{U_j-a_j}{b_j-a_j},\qquad j=1,\dots,n,
\]
are characteristic functions of arcs of lengths \(2\pi q_j\). The centers of
all these arcs determine the same unoriented line through the origin in the
parameter plane \(\R^2\). Conversely, every such choice gives equality.
\end{proposition}

\begin{proof}
Equality in \eqref{eq:box-bound} forces equality in every scalar interval
estimate and in the comparison of the operator norm with the Hilbert-Schmidt
norm. Theorem~\ref{thm:interval} therefore gives the asserted arc structure
of the boundary functions. Since every \(p_j\) lies in the interior of its
interval, all scalar constants are positive; Lemma~\ref{lem:operator-HS}
then shows that all row vectors \(\nabla f_j(0)\) are collinear. Conversely,
collinear component gradients attaining their scalar bounds make the operator
norm equal to the Hilbert-Schmidt norm and yield equality in
\eqref{eq:box-bound}.
\end{proof}

The unit hypercube is obtained by substituting the common interval
\((0,1)\) into the box formula.

\begin{corollary}\label{cor:hypercube}
For the unit hypercube \(Q=(0,1)^n\) and \(p=(p_1,\dots,p_n)\in Q\),
\[
 \sup\left\{
 \|dF_0\|:
 F:\D\to(0,1)^n\ \text{harmonic},\ F(0)=p
 \right\}
 =\frac2\pi\left(\sum_{j=1}^n \sin^2(\pi p_j)\right)^{1/2}.
\]
\end{corollary}

\begin{proof}
For \((0,1)^n\), Theorem~\ref{thm:box} gives \(A_j(p_j)=(2/\pi)\sin(\pi p_j)\). Substitution in \eqref{eq:MQ-def} gives the formula.
\end{proof}

The choice of norm on the target is essential to the square-root formula.

\begin{remark}\label{rem:target-norm}
The square-root formula in Theorem~\ref{thm:box} is tied to the Euclidean operator norm of \(dF_0:\R^2\to\R^n\). If a different target norm were chosen, the extremal formula would change. Decoupling the coordinate constraints alone therefore does not determine \(M_Q(p)\); the Euclidean geometry of the full differential is essential.
\end{remark}

\section{A product principle for prescribed-value extremals}
\label{sec:product-principle}

The preceding box calculation contains a product mechanism that is independent
of interval geometry. It uses only two facts: each component remains admissible
after a rotation of the parameter disk, and the squared Euclidean norm on the
product is the sum of the squared component norms. The same argument therefore
applies to arbitrary products of bounded Euclidean domains.

For a bounded domain \(G\subset\R^m\) and \(p\in G\), we use the notation
\[
 M_G(p)=
 \sup\{\|df_0\|:f:\D\to G\ \text{harmonic},\ f(0)=p\}.
\]
The number \(M_G(p)\) is finite because \(G\) is bounded. Indeed, after
placing \(G\) in a Euclidean ball, the Poisson derivative formula gives a
uniform first-derivative bound at the origin. It is positive because \(p\) is
an interior point: a sufficiently short affine harmonic segment through \(p\)
is contained in \(G\).

\begin{theorem}\label{thm:product-principle}
Let \(G_j\subset\R^{m_j}\) be bounded domains,
\(p_j\in G_j\), \(j=1,\ldots,N\), with
\[
 G=G_1\times\cdots\times G_N,\qquad
 p=(p_1,\ldots,p_N).
\]
The target product is equipped with its Euclidean norm. Then
\begin{equation}\label{eq:product-principle}
 M_G(p)^2
 =
 \sum_{j=1}^NM_{G_j}(p_j)^2.
\end{equation}
\end{theorem}

\begin{proof}
For a harmonic map
\(F=(f_1,\ldots,f_N):\D\to G\) with \(F(0)=p\),
Lemma~\ref{lem:block-operator} gives
\[
 \|dF_0\|^2
 \leq
 \sum_{j=1}^N\|d(f_j)_0\|^2
 \leq
 \sum_{j=1}^NM_{G_j}(p_j)^2.
\]
This proves the upper estimate.

For the reverse inequality, choose
\(0<\varepsilon<\min_j M_{G_j}(p_j)\). For each \(j\), there is a harmonic
map \(g_j:\D\to G_j\) with \(g_j(0)=p_j\) and
\[
 \|d(g_j)_0\|>M_{G_j}(p_j)-\varepsilon.
\]
A unit vector \(v_j\in\R^2\) attains the norm of \(d(g_j)_0\).  A rotation
\(R_j\) of the parameter disk can be chosen so that \(R_je_1=v_j\). The maps
\(f_j=g_j\circ R_j\) remain harmonic, satisfy \(f_j(0)=p_j\), and
\[
 |d(f_j)_0e_1|=\|d(g_j)_0\|.
\]
For \(F_\varepsilon=(f_1,\ldots,f_N)\),
\[
 \|d(F_\varepsilon)_0\|^2
 \geq
 |d(F_\varepsilon)_0e_1|^2
 =
 \sum_{j=1}^N\|d(g_j)_0\|^2
 >
 \sum_{j=1}^N\bigl(M_{G_j}(p_j)-\varepsilon\bigr)^2.
\]
Letting \(\varepsilon\to0^+\) gives the reverse inequality and proves
\eqref{eq:product-principle}.
\end{proof}

The equality cases admit a corresponding factorwise characterization.

\begin{theorem}
\label{thm:product-equality}
Under the hypotheses of Theorem~\ref{thm:product-principle}, let
\(F=(f_1,\ldots,f_N):\D\to G\) be harmonic with \(F(0)=p\).
Then
\[
 \|dF_0\|=M_G(p)
\]
if and only if the following two conditions hold:
\begin{enumerate}
\item each component is extremal,
\[
 \|d(f_j)_0\|=M_{G_j}(p_j),
 \qquad j=1,\ldots,N;
\]
\item the differentials \(d(f_j)_0\) have a common maximizing direction in
\(\R^2\), that is, there is a unit vector \(v\in\R^2\) such that
\[
 |d(f_j)_0v|=\|d(f_j)_0\|,
 \qquad j=1,\ldots,N.
\]
\end{enumerate}
\end{theorem}

\begin{proof}
The proof of Theorem~\ref{thm:product-principle} gives
\[
 \|dF_0\|^2
 \leq
 \sum_{j=1}^N\|d(f_j)_0\|^2
 \leq
 \sum_{j=1}^NM_{G_j}(p_j)^2
 =
 M_G(p)^2.
\]
If the first and last terms are equal, then both intermediate inequalities
are equalities. The second equality forces every component to be extremal.
The first equality is characterized by Lemma~\ref{lem:block-operator} and
gives the common maximizing direction. The converse follows from the same
lemma.
\end{proof}

\begin{remark}\label{rem:box-product-principle}
For \(G_j=(a_j,b_j)\), Theorem~\ref{thm:interval} gives
\(M_{G_j}(p_j)=A_j(p_j)\), so Theorem~\ref{thm:product-principle} recovers
Theorem~\ref{thm:box}. In the scalar-block case, the common maximizing
direction in
Theorem~\ref{thm:product-equality} is exactly the collinearity of the
component gradients in Proposition~\ref{prop:box-equality}.
\end{remark}

The following algebraic inequality relates the Euclidean and coordinatewise
product structures. If \(m_j>0\) and
\(v=(v_1,\ldots,v_N)\) belongs to a Euclidean product, then
\begin{equation}\label{eq:l2-linfty-comparison}
 \frac{2|v|}{(\sum_{j=1}^Nm_j^2)^{1/2}}
 \leq
 \max_{1\leq j\leq N}\frac{2|v_j|}{m_j}.
\end{equation}
Indeed,
\[
 \frac{4|v|^2}{\sum_{j=1}^N m_j^2}
 =
 \sum_{j=1}^N\frac{m_j^2}{\sum_{k=1}^N m_k^2}
 \left(\frac{2|v_j|}{m_j}\right)^2
 \leq
 \max_j\left(\frac{2|v_j|}{m_j}\right)^2.
\]
We shall use this comparison to relate the Euclidean extremal density to the
coordinatewise product metrics.

\section{Pointwise forms and complete conformal metrics on the domain}
\label{sec:pointwise-general}

The definition of \(M_G\) is local at the origin. Conformal
reparametrization by disk automorphisms turns the origin estimate into an
exact pointwise statement for every bounded Euclidean target.

\begin{theorem}
\label{thm:general-pointwise}
Let \(G\subset\R^m\) be a bounded domain and let
\(F:\D\to G\) be harmonic. Then
\begin{equation}\label{eq:general-pointwise}
 \|dF_z\|
 \leq
 \frac{M_G(F(z))}{1-|z|^2},
 \qquad z\in\D.
\end{equation}
For fixed \(z\in\D\) and \(p\in G\), the right-hand side is the exact
supremum among harmonic maps satisfying \(F(z)=p\).
\end{theorem}

\begin{proof}
For \(z\in\D\), a disk automorphism \(\varphi_z\) can be chosen with
\(\varphi_z(0)=z\) and \(|\varphi_z'(0)|=1-|z|^2\).  The map
\(F\circ\varphi_z\) is harmonic and takes the value \(F(z)\) at the origin.
Thus
\[
 (1-|z|^2)\|dF_z\|
 =
 \|d(F\circ\varphi_z)_0\|
 \leq
 M_G(F(z)).
\]
Conversely, near-extremals at the origin can be precomposed with
\(\varphi_z^{-1}\), which proves the sharpness statement.
\end{proof}

For product targets, Theorem~\ref{thm:product-principle} makes the
pointwise factor explicit.

\begin{corollary}\label{cor:product-pointwise}
Let \(G_j\subset\R^{m_j}\), \(j=1,\ldots,N\), be bounded domains, let
\(G=\prod_{j=1}^N G_j\) carry the Euclidean product norm, and let
\(F=(f_1,\ldots,f_N):\D\to G\) be harmonic. Then
\begin{equation}\label{eq:product-pointwise}
 \|dF_z\|
 \leq
 \frac{
 \left(\sum_{j=1}^NM_{G_j}(f_j(z))^2\right)^{1/2}
 }{1-|z|^2}.
\end{equation}
For a box \(Q=\prod(a_j,b_j)\), this becomes
\[
 \|dF_z\|
 \leq
 \frac{M_Q(F(z))}{1-|z|^2}.
\]
For a rectangle \(R\subset\C\), Lemma~\ref{lem:operator-norm} gives
\[
 |f_z(z)|+|f_{\overline z}(z)|
 \leq
 \frac{M_R(f(z))}{1-|z|^2}.
\]
\end{corollary}

The same argument is invariant under conformal reparametrization of the
domain. If \(\Omega\subsetneq\C\) is simply connected, we write
\(\lambda_\Omega(w)|dw|^2\) for its hyperbolic metric of curvature \(-1\).

\begin{corollary}\label{cor:general-hyperbolic-source}
Suppose \(\Omega\subsetneq\C\) is simply connected, \(G\subset\R^m\) is a
bounded domain, and \(F:\Omega\to G\) is harmonic. Then
\begin{equation}\label{eq:general-hyperbolic-source}
 \|dF_w\|
 \leq
 \frac{\sqrt{\lambda_\Omega(w)}}2\,M_G(F(w)),
 \qquad w\in\Omega.
\end{equation}
\end{corollary}

\begin{proof}
A conformal map \(h:\D\to\Omega\) with \(h(0)=w\) satisfies
\[
 \lambda_\Omega(w)|h'(0)|^2=4.
\]
Applying the definition of \(M_G(F(w))\) to \(F\circ h\) gives
\[
 |h'(0)|\,\|dF_w\|\leq M_G(F(w)),
\]
which is equivalent to \eqref{eq:general-hyperbolic-source}.
\end{proof}

The pointwise estimate also extends to complete conformal metrics on the
domain, independently of the shape of the target.
For a local expression \(ds^2=\rho(z)|dz|^2\), we use the curvature
normalization
\[
 K_\rho(z)=-\frac{1}{2\rho(z)}\Delta\log\rho(z),
 \qquad
 \Delta=4\partial_z\partial_{\overline z}.
\]
The resulting invariant estimate is stated next.

\begin{theorem}\label{thm:complete-general}
Let \(X\) be a hyperbolic Riemann surface and let
\(ds^2=\rho(z)|dz|^2\) be a complete \(C^2\) conformal metric on \(X\).
Suppose
\[
 K_\rho\geq-\kappa^2
\]
for some \(\kappa>0\). If \(G\subset\R^m\) is a bounded domain and
\(F:X\to G\) is harmonic in local conformal coordinates, then
\begin{equation}\label{eq:complete-general}
 \frac{\|dF_P\|}{\sqrt{\rho(P)}}
 \leq
 \frac{\kappa}{2}M_G(F(P)),
 \qquad P\in X.
\end{equation}
The quotient on the left is independent of the local conformal coordinate.
\end{theorem}

\begin{proof}
In a local conformal coordinate, let
\(\lambda_X(z)|dz|^2\) denote the hyperbolic metric of curvature \(-1\) on
\(X\). By the Schwarz comparison theorem for complete conformal metrics in
the \(C^2\) setting, applied to the holomorphic identity map
\[
 \operatorname{id}:(X,\rho|dz|^2)
 \longrightarrow (X,\lambda_X|dz|^2),
\]
the completeness of \(\rho|dz|^2\), together with
\(K_\rho\geq-\kappa^2\), gives
\begin{equation}\label{eq:metric-comparison-general}
 \lambda_X(P)\leq\kappa^2\rho(P).
\end{equation}
This is the one-dimensional form of Yau's Schwarz comparison principle
\cite{Yau}; for the \(C^2\) Riemann-surface formulation used here, see
\cite{KMS2025}. For the normalization of the hyperbolic metric, see
\cite{Ahlfors}.

A universal covering \(\pi:\D\to X\) with \(\pi(0)=P\) satisfies
\[
 \lambda_X(P)|\pi'(0)|^2=4.
\]
Applying the definition of \(M_G(F(P))\) to \(F\circ\pi\) yields
\[
 |\pi'(0)|\,\|dF_P\|\leq M_G(F(P)),
\]
hence
\[
 \|dF_P\|
 \leq
 \frac{\sqrt{\lambda_X(P)}}2M_G(F(P)).
\]
Combining this estimate with \eqref{eq:metric-comparison-general} gives
\eqref{eq:complete-general}. Under a change of local conformal coordinate,
both \(\|dF_P\|\) and \(\sqrt{\rho(P)}\) acquire the same conformal factor.
\end{proof}

The curvature factor in Theorem~\ref{thm:complete-general} is already sharp
in the constant-curvature model. On \(X=\D\), the metric
\(\rho=\lambda_\D/\kappa^2\) has curvature \(-\kappa^2\) and satisfies
\(\lambda_\D=\kappa^2\rho\).  At the origin, an extremal map for
\(M_G(p)\), or a sequence of near-extremals when the supremum is not
attained, makes \eqref{eq:complete-general} sharp.

For a product target \(G=\prod G_j\), the factor in
\eqref{eq:complete-general} is
\[
 M_G(F(P))
 =
 \left(
 \sum_{j=1}^NM_{G_j}(f_j(P))^2
 \right)^{1/2}.
\]
The same domain estimate therefore applies to both boxes and polydiscs once
their prescribed-value factor constants are inserted.

The exact operator-norm estimate also defines a Euclidean extremal density
whenever \(M_G\) is known:
\begin{equation}\label{eq:Euclidean-extremal-density}
 \mathcal E_G(p;v)=\frac{2|v|}{M_G(p)}.
\end{equation}
Corollary~\ref{cor:general-hyperbolic-source} is precisely the infinitesimal
contraction
\[
 \mathcal E_G(F(w);dF_w\xi)
 \leq
 \sqrt{\lambda_\Omega(w)}\,|\xi|.
\]
For product targets, this \(\ell^2\)-based density is distinct from the
coordinatewise \(\ell^\infty\)-based structure introduced below.

\section{The strip method, complexification, and the product-strip metric}
\label{sec:strip-product}

For interval factors, the sharp scalar density has an intrinsic hyperbolic
meaning. Identifying it with the Poincar\'e metric of a vertical strip places
the coordinatewise box metric inside classical complex geometry and explains
its completeness.

\subsection{The strip method}

For the symmetric interval \((-1,1)\), Theorem~\ref{thm:interval} reads
\begin{equation}\label{eq:strip-form}
 |\nabla u(0)|
 \leq
 \frac4\pi\cos\frac{\pi u(0)}2,
 \qquad
 u:\D\to(-1,1)\ \text{harmonic}.
\end{equation}
We use the strip
\[
 S=\{w\in\C:-1<\re w<1\}.
\]
The Poincar\'e metric of curvature \(-1\) on \(S\) is
\begin{equation}\label{eq:strip-metric}
 ds_S^2=
 \left(\frac\pi2\right)^2
 \frac{|dw|^2}{\cos^2(\frac\pi2\re w)}.
\end{equation}
Since \(\D\) is simply connected, \(u\) has a harmonic conjugate \(v\),
unique up to an additive constant. Thus \(H=u+iv\) is holomorphic from
\(\D\) to \(S\), and \(|H'(z)|=|\nabla u(z)|\).  Schwarz-Pick for \(H\),
with the curvature \(-1\) density \(2/(1-|z|^2)\) on \(\D\), gives
\[
 \frac{\pi}{2\cos(\frac\pi2u(z))}
 |\nabla u(z)|
 \leq
 \frac{2}{1-|z|^2}.
\]
At \(z=0\) this is \eqref{eq:strip-form}; an affine normalization gives
Theorem~\ref{thm:interval}.  Related strip and hyperbolic derivative methods
appear in
\cite{Burgeth,Chen2013,KalajVuorinen,Mateljevic2018,MateljevicSvetlik,KMS2025}.

For a general interval \((a_j,b_j)\), we associate the vertical strip
\[
 S_j=\{w\in\C:a_j<\re w<b_j\}.
\]
The Poincar\'e norm of curvature \(-1\) on \(S_j\) is
\begin{equation}\label{eq:stripj-hyp}
 \Hyp_{S_j}(w;\eta)
 =
 \frac{\pi|\eta|}
 {(b_j-a_j)
 \sin\!\left(\frac{\pi(\re w-a_j)}{b_j-a_j}\right)}.
\end{equation}
For \(p_j\in(a_j,b_j)\) and a real tangent vector \(v_j\),
\begin{equation}\label{eq:stripj-Aj}
 \Hyp_{S_j}(p_j;v_j)
 =
 \frac{2|v_j|}{A_j(p_j)}.
\end{equation}
Thus the sharp interval density is exactly the restriction of the strip
Poincar\'e metric to the real slice.

The real interval is in fact a hyperbolic geodesic. The conformal map
\[
 \Psi_j(w)=
 \exp\!\left(\frac{i\pi(w-a_j)}{b_j-a_j}\right)
\]
maps \(S_j\) onto the upper half-plane and \((a_j,b_j)\) onto the upper unit
semicircle. Hence, for \(x,y\in(a_j,b_j)\), the ambient strip distance is
\begin{equation}\label{eq:delta-j}
 \delta_j(x,y)
 =
 \left|
 \log\frac{
 \tan\!\left(\frac{\pi(y-a_j)}{2(b_j-a_j)}\right)}
 {\tan\!\left(\frac{\pi(x-a_j)}{2(b_j-a_j)}\right)}
 \right|.
\end{equation}

\subsection{Complexification of a box and the product metric}
\label{subsec:product-strip}

For
\[
 Q=\prod_{j=1}^n(a_j,b_j)
\]
we set
\[
 S_Q=\prod_{j=1}^nS_j.
\]
Then
\[
 Q=S_Q\cap\R^n.
\]
Thus \(S_Q\) is the coordinatewise complexification of \(Q\), obtained by
replacing each interval by the vertical strip having that interval as its
real geodesic.

For \(p\in Q\) and \(v\in T_pQ\simeq\R^n\), we define
\begin{equation}\label{eq:FQ-product}
 \mathcal F_Q(p;v)
 =
 \max_{1\leq j\leq n}\frac{2|v_j|}{A_j(p_j)}
 =
 \max_{1\leq j\leq n}\Hyp_{S_j}(p_j;v_j).
\end{equation}
This is a continuous family of reversible norms. Smoothness and strong
convexity are not required here.

We denote by \(d_Q\) the induced path distance. The factor geodesics can be
parametrized on a common interval with constant hyperbolic speeds, and the
max structure in \eqref{eq:FQ-product} gives
\begin{equation}\label{eq:dQ-product}
 d_Q(p,q)
 =
 \max_{1\leq j\leq n}\delta_j(p_j,q_j).
\end{equation}
The lower bound in \eqref{eq:dQ-product} follows because every curve from
\(p\) to \(q\) has length at least \(\delta_j(p_j,q_j)\) in its
\(j\)-th coordinate. For the reverse inequality, let
\[
 \Delta=\max_j\delta_j(p_j,q_j).
\]
For each \(j\), a coordinate geodesic from \(p_j\) to \(q_j\) can be
parametrized on \([0,1]\) with constant speed
\(\delta_j(p_j,q_j)\).  The product curve then has
\(\mathcal F_Q\)-speed \(\Delta\) almost everywhere and length \(\Delta\),
which proves \eqref{eq:dQ-product}.  Since each real geodesic
\((a_j,b_j)\) is complete in the induced strip distance, the max formula
also shows that \(d_Q\) is complete.

We use the standard Kobayashi normalization
\[
 \kappa_\D(z;\xi)=\frac{|\xi|}{1-|z|^2}.
\]
The product property of the Kobayashi-Royden metric and the Kobayashi
distance \cite{Kobayashi1998,JarnickiPflugProduct} gives
\[
 \kappa_{S_Q}(w;\eta)
 =
 \max_j\kappa_{S_j}(w_j;\eta_j),
 \qquad
 k_{S_Q}(w,\zeta)
 =
 \max_jk_{S_j}(w_j,\zeta_j).
\]
Since the curvature \(-1\) Poincar\'e norm is twice the standard
Kobayashi-Royden norm,
\begin{equation}\label{eq:Kob-restriction}
 \mathcal F_Q=2\kappa_{S_Q}|_{TQ},
 \qquad
 d_Q=2k_{S_Q}|_{Q\times Q}.
\end{equation}
For the symmetric cube \(Q=(-1,1)^n\),
\[
 \mathcal F_Q(p;v)
 =
 \max_{1\leq j\leq n}
 \frac{\pi|v_j|}{2\cos(\frac\pi2p_j)}.
\]
We denote by \(d_\D\) the Poincar\'e distance induced by the curvature
\(-1\) density \(2/(1-|z|^2)\). With the present normalization,
\(d_\D=2k_\D\).

The coordinate estimates give the following contraction theorem with an
optimal infinitesimal constant.

\begin{theorem}\label{thm:product-strip-contraction}
Let \(h:\D\to Q\) be harmonic. Then
\begin{equation}\label{eq:product-strip-infinitesimal}
 \mathcal F_Q(h(z);dh_z\xi)
 \leq
 \Hyp_\D(z;\xi)
 :=
 \frac{2|\xi|}{1-|z|^2}
\end{equation}
for every \(z\in\D\) and \(\xi\in T_z\D\). Consequently,
\begin{equation}\label{eq:product-strip-global}
 d_Q(h(z_0),h(z_1))
 \leq
 d_\D(z_0,z_1).
\end{equation}
The constant in \eqref{eq:product-strip-infinitesimal} is sharp.
\end{theorem}

\begin{proof}
For \(h=(h_1,\ldots,h_n)\), the pointwise interval estimate gives
\[
 \frac{2|dh_j(z)\xi|}{A_j(h_j(z))}
 \leq
 \frac{2|\xi|}{1-|z|^2}
\]
for every \(j\).  Taking the maximum proves
\eqref{eq:product-strip-infinitesimal}.  Integration along piecewise
\(C^1\) curves gives \eqref{eq:product-strip-global}.  Sharpness already
occurs when one component is a scalar interval extremal and the remaining
components are constant.
\end{proof}

The Euclidean extremal density from \eqref{eq:Euclidean-extremal-density}
and the product-strip metric satisfy
\begin{equation}\label{eq:Euclidean-product-comparison}
 \mathcal E_Q(p;v)
 =
 \frac{2|v|}{M_Q(p)}
 \leq
 \mathcal F_Q(p;v).
\end{equation}
This is \eqref{eq:l2-linfty-comparison} with
\(m_j=A_j(p_j)\).  The two structures arise from the same sharp factor estimates, but serve
different purposes: \(\mathcal E_Q\) is adapted to the Euclidean operator
norm of the full differential, while \(\mathcal F_Q\) is the coordinatewise
product metric.

Every \(S_j\) is biholomorphic to \(\D\), so \(S_Q\) is biholomorphic to the
polydisc \(\D^n\). The identification
\eqref{eq:Kob-restriction} is therefore the restriction to the real slice
of the standard Kobayashi product geometry of a polydisc. This biholomorphic equivalence
does not transfer the harmonic mapping problem itself, since postcomposition
by a nonlinear target biholomorphism need not preserve harmonicity. The
disk-valued factor problem must therefore be solved directly.

\section{The disk-valued prescribed-value factor}
\label{sec:disk-factor}

The unrestricted prescribed-value disk factor was determined in
\cite{KnezevicMateljevic2026}. We give a self-contained derivation here in
the form needed for the product theorem, its equality cases, and the metric
construction below. By target rotations, the factor depends only on
\(r=|p|\); we write
\[
 \mathcal M(r)=M_\D(p),\qquad |p|=r.
\]

If \(f=P[\Phi]\) maps \(\D\) harmonically into \(\D\), then
\(|\Phi|\leq1\) almost everywhere. After rotating the target and the parameter disk,
the prescribed-value problem becomes
\begin{equation}\label{eq:disk-reduction}
 \mathcal M(r)
 =
 2\sup\left\{
 \left|
 \int_\T \cos t\,\Phi(e^{it})\dm
 \right|:
 |\Phi|\leq1,\quad
 \int_\T\Phi\dm=r
 \right\},
 \qquad 0\leq r<1.
\end{equation}
Conversely, let \(\Phi\) satisfy the constraints in
\eqref{eq:disk-reduction}. By convexity, its Poisson integral takes values
in \(\overline\D\), and in fact in \(\D\). Indeed, if
\(f(z_0)=\omega\in\T\), then, with
\[
 P_{z_0}(e^{it})=\frac{1-|z_0|^2}{|e^{it}-z_0|^2}>0,
\]
positivity of the Poisson kernel gives
\[
 0=\int_\T P_{z_0}(e^{it})
 \bigl(1-\re(\overline\omega\Phi(e^{it}))\bigr)\dm.
\]
The integrand is nonnegative, hence \(\Phi=\omega\) almost everywhere,
contrary to \(\int_\T\Phi\dm=r<1\).

The direction of the vector
\(\int_\T\Phi(e^{it})\cos t\dm\) is part of the extremal problem. The
following support inequality controls all directions simultaneously.

\begin{lemma}\label{lem:disk-support}
For every \(\lambda\in\R\), every \(\eta\in\T\), every \(x\in\R\), and every
\(\zeta\in\overline\D\),
\begin{equation}\label{eq:disk-support}
 \re(\overline\eta\,x\zeta)
 \leq
 |\lambda+x\eta|-\lambda\re\zeta.
\end{equation}
\end{lemma}

\begin{proof}
Since \(|\zeta|\leq1\),
\[
 \re(\overline\eta\,x\zeta)+\lambda\re\zeta
 =
 \re((\lambda+\overline\eta\,x)\zeta)
 \leq
 |\lambda+\overline\eta\,x|
 =
 |\lambda+x\eta|.
\]
\end{proof}

For \(\eta\in\T\), we denote by \(S_\eta(r)\) the supremum of
\[
 \re\left(
 \overline\eta
 \int_\T\cos t\,\Phi(e^{it})\dm
 \right)
\]
over the admissible functions in \eqref{eq:disk-reduction}.  Thus
\[
 \mathcal M(r)=2\sup_{\eta\in\T}S_\eta(r).
\]
The next lemma gives a common upper bound for all target directions.
Theorem~\ref{thm:disk-main} will show that the perpendicular direction
attains this bound when \(r>0\).

\begin{lemma}
\label{lem:disk-direction}
For \(0\leq r<1\),
\begin{equation}\label{eq:disk-direction}
 \sup_{\eta\in\T}S_\eta(r)
 \leq
 \inf_{\lambda\in\R}
 \left\{
 \int_\T\sqrt{\lambda^2+\cos^2t}\dm-\lambda r
 \right\}.
\end{equation}
\end{lemma}

\begin{proof}
Lemma~\ref{lem:disk-support} gives
\[
 S_\eta(r)
 \leq
 \inf_{\lambda\in\R}
 \left\{
 \int_\T|\lambda+\eta\cos t|\dm-\lambda r
 \right\}.
\]
With \(a=\re\eta\), for \(x\geq0\),
\[
 |\lambda\pm x\eta|^2
 =
 \lambda^2+x^2\pm2\lambda xa.
\]
If \(A=\lambda^2+x^2\) and \(B=|2\lambda xa|\), then \(0\leq B\leq A\), and
\[
 B\longmapsto\sqrt{A+B}+\sqrt{A-B}
\]
is decreasing on \([0,A]\).  Hence
\[
 \frac{|\lambda+x\eta|+|\lambda-x\eta|}{2}
 \leq
 \sqrt{\lambda^2+x^2}.
\]
Pairing the values \(\cos t\) and \(-\cos t\) in the integral yields
\[
 \int_\T|\lambda+\eta\cos t|\dm
 \leq
 \int_\T\sqrt{\lambda^2+\cos^2t}\dm.
\]
Taking the infimum in \(\lambda\) and then the supremum in \(\eta\) proves
\eqref{eq:disk-direction}.
\end{proof}

The parameter in the sharp formula is determined by a strictly monotone
equation.

\begin{lemma}\label{lem:disk-lambda}
For every \(r\in(0,1)\), there is a unique \(\lambda(r)>0\) satisfying
\begin{equation}\label{eq:disk-lambda}
 r=
 \int_\T
 \frac{\lambda(r)}
 {\sqrt{\lambda(r)^2+\cos^2t}}
 \dm.
\end{equation}
Moreover, \(\lambda(r)\) is continuous and strictly increasing,
\(\lambda(r)\to0\) as \(r\to0^+\), and
\(\lambda(r)\to+\infty\) as \(r\to1^-\).
\end{lemma}

\begin{proof}
For \(\lambda>0\), consider
\[
 R(\lambda)=
 \int_\T
 \frac{\lambda}{\sqrt{\lambda^2+\cos^2t}}
 \dm.
\]
The function \(R\) is continuous by dominated convergence. For
\(\lambda>0\), differentiation under the integral sign gives
\[
 R'(\lambda)=
 \int_\T
 \frac{\cos^2t}{(\lambda^2+\cos^2t)^{3/2}}\dm>0,
\]
since the integrand is positive except at the two points where \(\cos t=0\).
Thus \(R\) is strictly increasing. Dominated convergence also gives
\(R(\lambda)\to0\) as \(\lambda\to0^+\) and
\(R(\lambda)\to1\) as \(\lambda\to+\infty\). Hence \(R\) maps
\((0,\infty)\) bijectively onto \((0,1)\), and its inverse \(\lambda(r)\)
is continuous and strictly increasing.
\end{proof}

The exact disk factor is given by the following theorem.

\begin{theorem}\label{thm:disk-main}
Let \(f:\D\to\D\) be harmonic and let \(|f(0)|=r\).  If \(r=0\), then
\begin{equation}\label{eq:M0}
 \|df_0\|\leq\frac4\pi,
\end{equation}
and the constant is sharp.  If \(0<r<1\), let \(\lambda=\lambda(r)\) be
defined by \eqref{eq:disk-lambda}. Then
\begin{equation}\label{eq:disk-main}
 \|df_0\|
 \leq
 \mathcal M(r)
 :=
 2\int_\T
 \frac{\cos^2t}
 {\sqrt{\lambda^2+\cos^2t}}
 \dm,
\end{equation}
and the bound is sharp.
\end{theorem}

\begin{proof}
The case \(r=0\) follows from \eqref{eq:complex-poisson-norm}:
\[
 \|df_0\|
 \leq
 2\int_\T|\cos t|\dm
 =
 \frac4\pi.
\]
Equality is attained by the Poisson integral of
\[
 \Phi(e^{it})=i\,\sgn(\cos t).
\]

For \(0<r<1\), consider
\[
 G(\mu)=
 \int_\T\sqrt{\mu^2+\cos^2t}\dm-\mu r.
\]
The function \(G\) is convex and, for \(\mu\ne0\),
\[
 G'(\mu)=
 \int_\T
 \frac{\mu}{\sqrt{\mu^2+\cos^2t}}\dm-r.
\]
For \(\mu<0\), this derivative is negative, while the right derivative at
\(0\) equals \(-r<0\).  On \((0,\infty)\), Lemma~\ref{lem:disk-lambda}
shows that \(G'(\mu)\) vanishes exactly at \(\mu=\lambda(r)\), is negative
before that point, and is positive afterwards.  Thus \(G\) has the unique
minimizer \(\lambda(r)>0\).  Lemma~\ref{lem:disk-direction} therefore gives
\[
 \frac{\mathcal M(r)}2
 \leq
 G(\lambda)
 =
 \int_\T
 \frac{\cos^2t}{\sqrt{\lambda^2+\cos^2t}}
 \dm.
\]
Here \(\mathcal M(r)\) denotes the supremum in
\eqref{eq:disk-reduction}.

For the reverse inequality, we use
\begin{equation}\label{eq:disk-extremal-boundary-real}
 \Phi_\lambda(e^{it})
 =
 \frac{\lambda+i\cos t}
 {\sqrt{\lambda^2+\cos^2t}}.
\end{equation}
Then \(|\Phi_\lambda|=1\) almost everywhere and, by symmetry together with
\eqref{eq:disk-lambda},
\[
 \int_\T\Phi_\lambda\dm=r.
\]
Its Poisson integral maps \(\D\) into \(\D\), because the boundary function
is unimodular and not almost everywhere constant. Moreover,
\[
 \int_\T\cos t\,\Phi_\lambda(e^{it})\dm
 =
 i\int_\T
 \frac{\cos^2t}{\sqrt{\lambda^2+\cos^2t}}\dm.
\]
Formula \eqref{eq:complex-poisson-norm} gives the reverse inequality and
hence equality in \eqref{eq:disk-main}.
\end{proof}

The argument also determines all equality cases.

\begin{proposition}\label{prop:disk-equality}
Let \(f:\D\to\D\) be harmonic.

If \(f(0)=0\), equality in \eqref{eq:M0} holds if and only if, up to a null
set, its radial boundary function has the form
\begin{equation}\label{eq:disk-zero-equality}
 \Phi(e^{it})
 =
 \omega\,\sgn(\cos(t-\alpha)),
 \qquad
 \omega\in\T,\quad \alpha\in\R.
\end{equation}

If \(f(0)=p=re^{i\beta}\ne0\), equality in
\eqref{eq:disk-main} holds if and only if, up to a null set,
\begin{equation}\label{eq:disk-equality}
 \Phi(e^{it})
 =
 e^{i\beta}
 \frac{\lambda(r)+i\cos(t-\alpha)}
 {\sqrt{\lambda(r)^2+\cos^2(t-\alpha)}}
\end{equation}
for some \(\alpha\in\R\).
\end{proposition}

\begin{proof}
For \(r=0\), choose \(\alpha\in\R\) so that
\[
 \|df_0\|=|df_0(e^{i\alpha})|.
\]
Such an \(\alpha\) exists by \eqref{eq:complex-poisson-norm}. Equality in
\eqref{eq:M0} then forces equality in
\[
 \left|
 \int_\T\Phi(e^{it})\cos(t-\alpha)\dm
 \right|
 \leq
 \int_\T|\cos(t-\alpha)|\dm.
\]
Hence \(|\Phi|=1\) almost everywhere and
\(\Phi(e^{it})\sgn(\cos(t-\alpha))\) has constant argument almost everywhere.
This gives \eqref{eq:disk-zero-equality}.

For \(0<r<1\), a target rotation reduces the problem to \(p=r\), and a
rotation of the parameter disk can be used so that
\(\|df_0\|=|df_0(e_1)|\). With
\[
 J=\int_\T\cos t\,\Phi(e^{it})\dm,
 \qquad \eta=J/|J|,
 \qquad \lambda=\lambda(r),
\]
the extremality assumption gives \(2|J|=\mathcal M(r)>0\).
Lemma~\ref{lem:disk-support} and the paired comparison in
Lemma~\ref{lem:disk-direction} then give
\[
 \frac{\mathcal M(r)}2
 \leq \int_\T|\lambda+\eta\cos t|\dm-\lambda r
 \leq \int_\T\sqrt{\lambda^2+\cos^2t}\dm-\lambda r
 =\frac{\mathcal M(r)}2.
\]
If \(\re\eta\ne0\), the paired comparison is strict whenever
\(|\cos t|>0\), so the second integral inequality is strict. Hence
\(\eta=\pm i\). Replacing the parameter \(t\) by \(t+\pi\) allows
\(\eta=i\). Equality in the first inequality now gives
\[
 \int_\T\left(
 \sqrt{\lambda^2+\cos^2t}
 -\re\bigl((\lambda-i\cos t)\Phi(e^{it})\bigr)
 \right)\dm=0.
\]
The integrand is nonnegative, so it vanishes almost everywhere. Since
\(\lambda-i\cos t\ne0\) and \(|\Phi|\leq1\), this forces
\[
 \Phi(e^{it})
 =\frac{\lambda+i\cos t}{\sqrt{\lambda^2+\cos^2t}}
\]
almost everywhere. Restoring the rotations gives
\eqref{eq:disk-equality}. The converse follows from the construction in
Theorem~\ref{thm:disk-main}.
\end{proof}

For the extremal in \eqref{eq:disk-equality}, the Poisson derivative formulas
give
\begin{equation}\label{eq:disk-extremal-differential}
 df_0(e^{i\theta})
 =i e^{i\beta}\mathcal M(r)\cos(\theta-\alpha).
\end{equation}
Indeed, the boundary function is even in \(t-\alpha\), so the derivative in
the perpendicular direction vanishes. For \eqref{eq:disk-zero-equality}, the
corresponding formula is
\begin{equation}\label{eq:disk-zero-extremal-differential}
 df_0(e^{i\theta})
 =\frac4\pi\omega\cos(\theta-\alpha).
\end{equation}
Thus in both cases the unique maximizing unoriented line is
\(\R e^{i\alpha}\).

The sharp constant can also be expressed in terms of complete elliptic integrals. We use the standard notation
\[
 K(k)=\int_0^{\pi/2}\frac{d\theta}{\sqrt{1-k^2\sin^2\theta}},
 \qquad
 E(k)=\int_0^{\pi/2}\sqrt{1-k^2\sin^2\theta}\,d\theta,
 \qquad 0<k<1.
\]

\begin{proposition}\label{prop:disk-elliptic}
For \(0<r<1\), let \(\lambda=\lambda(r)\) and
\[
 k=(1+\lambda^2)^{-1/2}.
\]
Then
\begin{equation}\label{eq:elliptic-lambda}
 r=
 \frac{2\lambda}{\pi\sqrt{1+\lambda^2}}K(k),
\end{equation}
and
\begin{equation}\label{eq:elliptic-M}
 \mathcal M(r)
 =
 \frac{4\sqrt{1+\lambda^2}}{\pi}
 \left(E(k)-(1-k^2)K(k)\right).
\end{equation}
\end{proposition}

\begin{proof}
Let \(a=\sqrt{1+\lambda^2}\) and \(k=a^{-1}\). Since
\[
 \lambda^2+\cos^2t=a^2(1-k^2\sin^2t),
\]
symmetry gives
\[
 \int_\T\frac{\dm}{\sqrt{\lambda^2+\cos^2t}}
 =\frac{2}{\pi a}K(k),
 \qquad
 \int_\T\sqrt{\lambda^2+\cos^2t}\dm
 =\frac{2a}{\pi}E(k).
\]
The first identity gives \eqref{eq:elliptic-lambda}. Using
\(\cos^2t=(\lambda^2+\cos^2t)-\lambda^2\), we obtain
\[
 \mathcal M(r)
 =\frac{4a}{\pi}E(k)-\frac{4\lambda^2}{\pi a}K(k).
\]
Since \(\lambda^2/a^2=1-k^2\), this is
\eqref{eq:elliptic-M}.
\end{proof}

The following explicit boundary function gives the lower bound needed for the
metric comparison.

\begin{lemma}\label{lem:disk-lower}
For \(0\leq r<1\),
\begin{equation}\label{eq:disk-lower}
 \mathcal M(r)
 \geq
 \frac4\pi\sqrt{1-r^2}.
\end{equation}
In particular,
\begin{equation}\label{eq:disk-vs-holo}
 \mathcal M(r)>1-r^2.
\end{equation}
\end{lemma}

\begin{proof}
The boundary function
\[
 \Phi_r(e^{it})
 =
 r+i\sqrt{1-r^2}\,\sgn(\cos t)
\]
is unimodular and has mean \(r\).  Its Poisson integral is therefore an
admissible harmonic self-map of \(\D\), and
\[
 \left|
 \int_\T\cos t\,\Phi_r(e^{it})\dm
 \right|
 =
 \sqrt{1-r^2}\int_\T|\cos t|\dm
 =
 \frac2\pi\sqrt{1-r^2}.
\]
Formula \eqref{eq:complex-poisson-norm} proves
\eqref{eq:disk-lower}.  Since \(0<\sqrt{1-r^2}\leq1<4/\pi\), multiplication by
\(\sqrt{1-r^2}\) gives
\[
 \frac4\pi\sqrt{1-r^2}>1-r^2.
\]
\end{proof}

The exact formula also yields the boundary asymptotics
\[
 \mathcal M(r)
 \sim
 2\sqrt{1-r}
 \sim
 \sqrt2\sqrt{1-r^2},
 \qquad r\to1^-.
\]
Indeed, as \(\lambda\to+\infty\), uniformly in \(t\),
\[
 \frac{\lambda}{\sqrt{\lambda^2+\cos^2t}}
 =1-\frac{\cos^2t}{2\lambda^2}+O(\lambda^{-4}),
\]
so \eqref{eq:disk-lambda} and
\(\int_\T\cos^2t\dm=1/2\) give
\[
 r=1-\frac{1}{4\lambda^2}+O(\lambda^{-4}).
\]
Similarly,
\[
 \mathcal M(r)
 =\frac1\lambda+O(\lambda^{-3}),
\]
and the stated asymptotics follow.

\section{Polydisc targets and the sharp harmonic product metric}
\label{sec:polydisc}

Once the disk factor is known, the product principle gives the sharp polydisc
constant without any further target-specific extremal calculation. Throughout
this section, the operator norm of the full differential uses the standard
Euclidean norm on \(\C^n\simeq\R^{2n}\); the corresponding
\(\ell^\infty\)-type product structure appears separately in the metric
statements below.

\begin{theorem}
\label{thm:polydisc}
Let
\[
 F=(f_1,\ldots,f_n):\D\to\D^n
\]
be harmonic and let
\(F(0)=p=(p_1,\ldots,p_n)\). Then
\begin{equation}\label{eq:polydisc-main}
 \|dF_0\|
 \leq
 M_{\D^n}(p)
 :=
 \left(
 \sum_{j=1}^n\mathcal M(|p_j|)^2
 \right)^{1/2},
\end{equation}
and the constant is sharp for every \(p\in\D^n\).

Equality holds if and only if every component \(f_j\) is extremal in the
disk-valued problem at \(p_j\) and all differentials \(d(f_j)_0\) attain
their norms on a common unit vector in \(\R^2\).
\end{theorem}

\begin{proof}
The formula and the equality statement are
Theorems~\ref{thm:product-principle} and
\ref{thm:product-equality} applied to \(n\) copies of the unit disk, together
with Theorem~\ref{thm:disk-main}.
\end{proof}

The equality family can be written explicitly. A common maximizing line in
\(\R^2\) is represented by \(\alpha\in\R\). If
\(p_j=r_je^{i\beta_j}\ne0\), the radial boundary function of the \(j\)-th
component has the form
\begin{equation}\label{eq:polydisc-boundary-nonzero}
 \Phi_j(e^{it})
 =
 e^{i\beta_j}
 \frac{\lambda(r_j)+i\sigma_j\cos(t-\alpha)}
 {\sqrt{\lambda(r_j)^2+\cos^2(t-\alpha)}},
 \qquad \sigma_j\in\{-1,1\}.
\end{equation}
Proposition~\ref{prop:disk-equality} gives an individual parameter
\(\alpha_j\) for each nonzero component.  The common-line condition in
Theorem~\ref{thm:polydisc} is equivalent to
\(\alpha_j\equiv\alpha\pmod{\pi}\), which is precisely the freedom
represented by \(\sigma_j\).  If \(p_j=0\), the corresponding boundary function
has the form
\begin{equation}\label{eq:polydisc-boundary-zero}
 \Phi_j(e^{it})
 =
 \omega_j\,\sgn(\cos(t-\alpha)),
 \qquad \omega_j\in\T,
\end{equation}
where a change of sign is absorbed into \(\omega_j\).  Conversely, the
Poisson extension of every boundary function obtained by combining
\eqref{eq:polydisc-boundary-nonzero} and
\eqref{eq:polydisc-boundary-zero} is extremal.  Thus every polydisc extremal has radial boundary values in the distinguished
boundary
\[
 \partial_s\D^n=\T^n
\]
almost everywhere.

The common maximizing line also determines the differential of every
extremal polydisc map.

\begin{corollary}\label{cor:polydisc-one-dimensional}
If \(F:\D\to\D^n\) attains equality in
\eqref{eq:polydisc-main}, then there are \(\alpha\in\R\) and a nonzero vector
\(W=(W_1,\ldots,W_n)\in\C^n\) such that
\begin{equation}\label{eq:polydisc-one-dimensional}
 dF_0(e^{i\theta})=W\cos(\theta-\alpha),
 \qquad \theta\in\R,
\end{equation}
and
\[
 |W_j|=\mathcal M(|p_j|),
 \qquad j=1,\ldots,n.
\]
Consequently, the real differential \(dF_0\) has one-dimensional image.
\end{corollary}

\begin{proof}
For \(p_j\ne0\), formula \eqref{eq:disk-extremal-differential} together with
\eqref{eq:polydisc-boundary-nonzero} gives
\[
 d(f_j)_0(e^{i\theta})
 =i\sigma_j e^{i\beta_j}\mathcal M(|p_j|)
  \cos(\theta-\alpha).
\]
For \(p_j=0\), formula \eqref{eq:disk-zero-extremal-differential} gives the
same form with a coefficient of modulus \(4/\pi=\mathcal M(0)\). Collecting
the component coefficients gives \eqref{eq:polydisc-one-dimensional}. Since
all \(\mathcal M(|p_j|)\) are positive, \(W\ne0\), and the image of
\(dF_0\) is exactly the real line spanned by \(W\).
\end{proof}

For box extremals, Proposition~\ref{prop:box-equality} gives the same
conclusion through the collinearity of the component gradients. The boundary
geometry is different, but the common-direction condition in the parameter
plane is the same.

At the center, the formula is especially simple.

\begin{corollary}\label{cor:polydisc-center}
For harmonic \(F:\D\to\D^n\) with \(F(0)=0\),
\begin{equation}\label{eq:polydisc-center}
 \|dF_0\|\leq\frac{4\sqrt n}{\pi},
\end{equation}
and the constant is sharp.
\end{corollary}

The pointwise form and the version for complete conformal metrics on the domain follow from
Section~\ref{sec:pointwise-general}.

\begin{corollary}\label{cor:polydisc-pointwise}
For harmonic \(F=(f_1,\ldots,f_n):\D\to\D^n\),
\begin{equation}\label{eq:polydisc-pointwise}
 \|dF_z\|
 \leq
 \frac{
 \left(
 \sum_{j=1}^n\mathcal M(|f_j(z)|)^2
 \right)^{1/2}
 }{1-|z|^2}.
\end{equation}
If \(X\) is a hyperbolic Riemann surface carrying a complete
\(C^2\) conformal metric \(\rho|dz|^2\) with
\(K_\rho\geq-\kappa^2\) for some \(\kappa>0\), then every harmonic
\(F:X\to\D^n\) satisfies
\begin{equation}\label{eq:polydisc-complete}
 \frac{\|dF_P\|}{\sqrt{\rho(P)}}
 \leq
 \frac{\kappa}{2}
 \left(
 \sum_{j=1}^n\mathcal M(|f_j(P)|)^2
 \right)^{1/2}.
\end{equation}
\end{corollary}

The coordinatewise sharp estimate defines an \(\ell^\infty\)-type product
metric. On the disk we set
\begin{equation}\label{eq:rho-h}
 \rho_h(w)=\frac{2}{\mathcal M(|w|)}.
\end{equation}
Harmonic densities and the associated distance-decreasing viewpoint were
considered earlier in \cite{Mateljevic2018,ArsenovicGajicMateljevic}. The density in \eqref{eq:rho-h} is defined by the exact prescribed-value
extremal function \(\mathcal M\), and is therefore sharp for the present
problem. The function \(\mathcal M\) is positive and continuous on \([0,1)\); this
follows from Lemma~\ref{lem:disk-lambda}, Theorem~\ref{thm:disk-main}, and
dominated convergence at \(r=0\).  On the polydisc we define
\begin{equation}\label{eq:H-polydisc}
 \mathcal H_{\D^n}(p;v)
 =
 \max_{1\leq j\leq n}
 \frac{2|v_j|}{\mathcal M(|p_j|)}.
\end{equation}
This is a continuous family of reversible norms on the tangent spaces of
\(\D^n\), analogous to \(\mathcal F_Q\).  We denote by \(d_h\) the path
distance induced by \(\rho_h(w)|dw|\), and by \(d_{\mathcal H}\) the path
distance induced by \(\mathcal H_{\D^n}\). The maximum in the infinitesimal
norm gives the corresponding maximum formula for the induced path distance.

\begin{proposition}\label{prop:H-product-distance}
For \(p,q\in\D^n\),
\begin{equation}\label{eq:H-product-distance}
 d_{\mathcal H}(p,q)
 =
 \max_{1\leq j\leq n}d_h(p_j,q_j).
\end{equation}
\end{proposition}

\begin{proof}
Every curve in \(\D^n\) has \(\mathcal H_{\D^n}\)-length at least the
\(\rho_h\)-length of each coordinate curve, which gives the lower bound.
For the reverse inequality, rectifiable coordinate curves whose lengths are
within \(\varepsilon\) of the corresponding \(d_h\)-distances may be
parametrized with constant speed on \([0,1]\).  Their product then has
length equal to the maximum of the coordinate lengths.  Letting \(\varepsilon\to0\) proves
\eqref{eq:H-product-distance}.
\end{proof}

The pointwise disk estimate gives the following contraction theorem for the
harmonic product metric.

\begin{theorem}
\label{thm:H-contraction}
Every harmonic map \(F:\D\to\D^n\) satisfies
\begin{equation}\label{eq:H-infinitesimal}
 \mathcal H_{\D^n}(F(z);dF_z\xi)
 \leq
 \Hyp_\D(z;\xi)
 =
 \frac{2|\xi|}{1-|z|^2}.
\end{equation}
Consequently,
\begin{equation}\label{eq:H-global}
 d_{\mathcal H}(F(z_0),F(z_1))
 \leq
 d_\D(z_0,z_1).
\end{equation}
The infinitesimal constant is sharp.
\end{theorem}

\begin{proof}
For every component,
Theorem~\ref{thm:general-pointwise} gives
\[
 |d(f_j)_z\xi|
 \leq
 \frac{\mathcal M(|f_j(z)|)}{1-|z|^2}|\xi|.
\]
Multiplication by \(2/\mathcal M(|f_j(z)|)\) and maximization in \(j\)
gives \eqref{eq:H-infinitesimal}.  Integration along curves gives
\eqref{eq:H-global}.  Sharpness occurs already with one nonconstant
component.
\end{proof}

The sharpness of \(\mathcal H_{\D^n}\) has the following precise pointwise
form within the class of coordinatewise norms considered below.

\begin{proposition}\label{prop:H-maximal-coordinatewise}
Suppose
\[
 \mathcal N(p;v)=\max_{1\leq j\leq n}a_j(p)|v_j|,
 \qquad a_j(p)>0,
\]
and assume that every harmonic map \(F:\D\to\D^n\) satisfies
\[
 \mathcal N(F(z);dF_z\xi)
 \leq \frac{2|\xi|}{1-|z|^2}.
\]
Then
\[
 \mathcal N(p;v)\leq\mathcal H_{\D^n}(p;v)
\]
for every \(p\in\D^n\) and \(v\in\C^n\).
\end{proposition}

\begin{proof}
For fixed \(p\in\D^n\) and an index \(j\), a disk extremal with prescribed
value \(p_j\) in the \(j\)-th component, together with constant remaining
components at the corresponding coordinates of \(p\), defines a harmonic
map into \(\D^n\). At the origin, in a unit direction maximizing the
differential of the nonconstant component, the assumed contraction gives
\[
 a_j(p)\mathcal M(|p_j|)\leq2.
\]
Hence \(a_j(p)\leq2/\mathcal M(|p_j|)\) for every \(j\), which proves the
claim after taking the maximum over the coordinates.
\end{proof}

Thus \(\mathcal H_{\D^n}\) is pointwise maximal in this coordinatewise class;
no maximality among arbitrary direction-dependent norms is asserted.

As in the box case, the Euclidean extremal density and the product density
are related by the \(\ell^2\)-to-\(\ell^\infty\) comparison:
\begin{equation}\label{eq:E-vs-H-polydisc}
 \mathcal E_{\D^n}(p;v)
 =
 \frac{2|v|}
 {\left(\sum_j\mathcal M(|p_j|)^2\right)^{1/2}}
 \leq
 \mathcal H_{\D^n}(p;v).
\end{equation}

The relation with the classical Kobayashi geometry differs from the box case.

\begin{theorem}
\label{thm:H-vs-Kob}
For every \(p\in\D^n\) and every nonzero \(v\in\C^n\),
\begin{equation}\label{eq:H-vs-Kob}
 \mathcal H_{\D^n}(p;v)
 <
 2\kappa_{\D^n}(p;v)
 =
 \max_j\frac{2|v_j|}{1-|p_j|^2}.
\end{equation}
Moreover, the path metric \(d_{\mathcal H}\) is not complete.
\end{theorem}

\begin{proof}
Lemma~\ref{lem:disk-lower} gives
\[
 \mathcal M(|p_j|)>1-|p_j|^2
\]
for every coordinate.  Hence every nonzero coordinate term in
\eqref{eq:H-polydisc} is strictly smaller than the corresponding Kobayashi
term, and \eqref{eq:H-vs-Kob} follows.

For incompleteness, Lemma~\ref{lem:disk-lower} yields
\[
 \rho_h(r)
 =
 \frac2{\mathcal M(r)}
 \leq
 \frac{\pi}{2\sqrt{1-r^2}},
 \qquad 0\leq r<1.
\]
Thus the radial distance from \(0\) to the Euclidean boundary is finite:
\[
 \int_0^1\rho_h(r)\,dr
 \leq
 \frac\pi2
 \int_0^1\frac{dr}{\sqrt{1-r^2}}
 =
 \frac{\pi^2}{4}.
\]
A radial sequence tending to \(1\) is therefore Cauchy in \(d_h\).
Since \(\rho_h\) is continuous and strictly positive on \(\D\), the path
metric \(d_h\) induces the Euclidean topology locally. Hence such a radial
sequence cannot converge in \(d_h\) to a point of \(\D\). Thus \(d_h\) is
not complete. By \eqref{eq:H-product-distance}, the same conclusion holds
for \(d_{\mathcal H}\) on \(\D^n\).
\end{proof}

\section{Duality, Pontryagin's principle, and boundary geometry}
\label{sec:duality-control}

Duality and optimal control give a complementary view of the extremal boundary
functions. They explain why interval and disk factors select different types
of boundary values, while the product theorem supplies the common-direction
condition that couples the factors.

\subsection{The interval problem and bang-bang extremals}

For \(q\in(0,1)\), the normalized interval boundary problem is
\begin{equation}\label{eq:dual-interval-problem}
 \sup\left\{
 \int_\T U(e^{it})\cos t\dm:
 0\leq U\leq1\ \text{a.e.},
 \quad
 \int_\T U\dm=q
 \right\}.
\end{equation}
For a real multiplier \(\mu\),
\[
 \int_\T U\cos t\dm
 =
 \mu q+\int_\T U(\cos t-\mu)\dm.
\]
Pointwise maximization over \(0\leq U\leq1\) gives
\begin{equation}\label{eq:interval-dual}
 \int_\T U\cos t\dm
 \leq
 \mu q+\int_\T(\cos t-\mu)_+\dm.
\end{equation}
The choice \(\mu=\cos(\pi q)\) makes the threshold set
\[
 \{\cos t>\mu\}
\]
an arc of normalized measure \(q\).  Equality in
\eqref{eq:interval-dual} then forces
\[
 U=1\quad\text{where }\cos t>\mu,
 \qquad
 U=0\quad\text{where }\cos t<\mu.
\]
The level set has measure zero, so the unique maximizer, up to a null set,
is the characteristic function of the arc \(|t|<\pi q\). Replacing the
weight by \(\cos(t-\theta)\) rotates this arc by \(\theta\), as in
Lemma~\ref{lem:rearrangement}.

The same calculation is a fixed-time control problem.  A measurable control
\(u:[0,2\pi]\to[0,1]\) determines
\[
 x'(t)=\frac{u(t)}{2\pi},
 \qquad
 y'(t)=\frac{u(t)\cos t}{2\pi},
 \qquad
 x(0)=y(0)=0,
\]
with terminal constraint \(x(2\pi)=q\) and objective
\(\max y(2\pi)\).  If \(\lambda_0\geq0\) is the multiplier of the terminal payoff and
\(\lambda_x,\lambda_y\) are the adjoint variables, the Pontryagin Hamiltonian
is
\[
 H_P(t,u)
 =
 \frac{u}{2\pi}
 (\lambda_x+\lambda_y\cos t).
\]
The adjoint variables are constant.  If \(\nu\) denotes the multiplier of
the terminal constraint, the augmented terminal payoff is
\[
 \lambda_0y(2\pi)+\nu\bigl(x(2\pi)-q\bigr),
\]
and the transversality condition gives
\(\lambda_y=\lambda_0\) and \(\lambda_x=\nu\).  In the abnormal case
\(\lambda_0=0\), nontriviality forces \(\lambda_x\ne0\), and the maximum
condition would make \(u\) identically \(0\) or \(1\), contrary to
\(0<q<1\).  Hence every maximizer is normal.  After the normalization
\(\lambda_0=1\) and the notation \(\lambda_x=-\mu\), Pontryagin's maximum
condition becomes
\[
 u(t)\in
 \operatorname*{argmax}_{0\leq v\leq1}
 v(\cos t-\mu).
\]
The terminal constraint again gives
\(\mu=\cos(\pi q)\).  Thus Pontryagin's principle reproduces the same
bang-bang extremal; see \cite{Pontryagin}.

For boxes, the control problem separates coordinatewise.  The bang-bang law
therefore recovers the vertex-valued boundary behavior in
Remark~\ref{rem:rectangle-boundary-values} and
Proposition~\ref{prop:box-equality}.  Theorem~\ref{thm:product-equality}
provides the additional common maximizing direction in \(\R^2\) that
couples the factors.

\subsection{The disk problem and unimodular extremals}

For a nonzero prescribed disk value, the control set is the closed unit disk
rather than an interval. After the reductions in Section~\ref{sec:disk-factor},
so that \(0<r<1\), the extremal direction is perpendicular to the prescribed
value. With \(c(t)=\cos t\) and the multiplier
\(\lambda=\lambda(r)>0\), the pointwise Hamiltonian
is
\begin{equation}\label{eq:disk-Hamiltonian}
 H_\lambda(t,\zeta)
 =
 \re((\lambda-i c(t))\zeta),
 \qquad |\zeta|\leq1.
\end{equation}
Since \(\lambda-i c(t)\ne0\), the maximizer is unique:
\begin{equation}\label{eq:disk-control}
 \zeta_\lambda(t)
 =
 \frac{\lambda+i c(t)}
 {\sqrt{\lambda^2+c(t)^2}}.
\end{equation}
The prescribed mean determines \(\lambda\) through
\eqref{eq:disk-lambda}. Thus the same pointwise linear maximization that yields
bang-bang interval controls yields a smooth unimodular control for a nonzero
prescribed disk value.

The difference comes from the geometry of the admissible control set. A nonzero linear functional on an
interval is maximized at an endpoint, whereas a nonzero real linear
functional on the Euclidean disk is maximized at a unique point of the unit
circle.  The factorwise product principle then gives the corresponding
boundary geometry:
\[
 \text{box extremals}\quad\longrightarrow\quad
 \text{vertices},
\]
while
\[
 \text{polydisc extremals}\quad\longrightarrow\quad
 \text{the distinguished boundary }\T^n.
\]
The common maximizing direction in \(\R^2\) is not a feature of the
pointwise control set; it is the equality condition imposed by the Euclidean
operator norm of the full product differential.

\section{Concluding remarks}\label{sec:conclusion}

The product principle isolates the part of the prescribed-value extremal
problem that is universal under finite Euclidean products. All information
specific to a factor is contained in the quantities \(M_{G_j}(p_j)\); on the
product, these constants combine by the Euclidean sum-of-squares law. For a
given product map, equality is governed by one additional condition: the
factor differentials must share a maximizing direction in the parameter
plane. Thus a sharp factor estimate immediately produces a sharp product
estimate, while the equality problem reduces to an explicit compatibility
condition.

The box and polydisc applications show that this mechanism is independent of
the boundary geometry of the factors. Box extremals are built from
endpoint-valued interval boundary functions and lead to a complete metric
inherited from the Poincar\'e geometry of vertical strips. Polydisc extremals
are built from unimodular disk boundary functions and lead to a sharp
harmonic product metric that is pointwise maximal within the coordinatewise
class but incomplete. The contrast separates the universal product principle
from the distinct metric geometries carried by the individual factors.

\begingroup
\raggedright
\bibliographystyle{plain}
\bibliography{PP_MK_MM}

@book{Ahlfors,
  author    = {Ahlfors, Lars V.},
  title     = {Conformal Invariants: Topics in Geometric Function Theory},
  publisher = {McGraw-Hill},
  address   = {New York},
  year      = {1973},
  note      = {AMS Chelsea reprint, Providence, RI, 2010; \doi{10.1090/chel/371}}
}

@article{ArsenovicGajicMateljevic,
  author  = {Arsenovi{\'c}, M. and Gaji{\'c}, J. and Mateljevi{\'c}, M.},
  title   = {{Schwarz-Pick} lemma for harmonic and generalized harmonic functions},
  journal = {Lobachevskii J. Math.},
  volume  = {45},
  number  = {12},
  year    = {2024},
  pages   = {5975--6010},
  note    = {\doi{10.1134/S1995080224607409}}
}

@book{AxlerBourdonRamey,
  author    = {Axler, S. and Bourdon, P. and Ramey, W.},
  title     = {Harmonic Function Theory},
  edition   = {Second},
  series    = {Graduate Texts in Mathematics},
  volume    = {137},
  publisher = {Springer},
  address   = {New York},
  year      = {2001},
  note      = {\doi{10.1007/978-1-4757-8137-3}}
}

@article{Burgeth,
  author  = {Burgeth, B.},
  title   = {A {Schwarz} lemma for harmonic and hyperbolic-harmonic functions in higher dimensions},
  journal = {Manuscripta Math.},
  volume  = {77},
  number  = {2-3},
  year    = {1992},
  pages   = {283--291},
  note    = {\doi{10.1007/BF02567058}}
}

@article{Chen2013,
  author  = {Chen, H. H.},
  title   = {The {Schwarz-Pick} lemma and {Julia} lemma for real planar harmonic mappings},
  journal = {Sci. China Math.},
  volume  = {56},
  number  = {11},
  year    = {2013},
  pages   = {2327--2334},
  note    = {\doi{10.1007/s11425-013-4691-0}}
}

@article{ChenRasila,
  author  = {Chen, S. and Rasila, A.},
  title   = {{Schwarz-Pick} type estimates of pluriharmonic mappings in the unit polydisk},
  journal = {Illinois J. Math.},
  volume  = {58},
  number  = {4},
  year    = {2014},
  pages   = {1015--1024},
  note    = {\doi{10.1215/ijm/1446819298}}
}

@article{ChenHamada,
  author  = {Chen, S. and Hamada, H.},
  title   = {Some sharp {Schwarz-Pick} type estimates and their applications of harmonic and pluriharmonic functions},
  journal = {J. Funct. Anal.},
  volume  = {282},
  number  = {1},
  year    = {2022},
  note    = {Art. 109254; \doi{10.1016/j.jfa.2021.109254}}
}

@article{Colonna,
  author  = {Colonna, F.},
  title   = {The {Bloch} constant of bounded harmonic mappings},
  journal = {Indiana Univ. Math. J.},
  volume  = {38},
  number  = {4},
  year    = {1989},
  pages   = {829--840},
  note    = {\doi{10.1512/iumj.1989.38.38039}}
}

@book{DurenBook,
  author    = {Duren, P.},
  title     = {Harmonic Mappings in the Plane},
  series    = {Cambridge Tracts in Mathematics},
  volume    = {156},
  publisher = {Cambridge University Press},
  address   = {Cambridge},
  year      = {2004},
  note      = {\doi{10.1017/CBO9780511546600}}
}

@article{DurenSchoberVariational,
  author  = {Duren, P. and Schober, G.},
  title   = {A variational method for harmonic mappings onto convex regions},
  journal = {Complex Variables Theory Appl.},
  volume  = {9},
  number  = {2-3},
  year    = {1987},
  pages   = {153--168},
  note    = {\doi{10.1080/17476938708814259}}
}

@article{DurenSchober,
  author  = {Duren, P. and Schober, G.},
  title   = {Linear extremal problems for harmonic mappings of the disk},
  journal = {Proc. Amer. Math. Soc.},
  volume  = {106},
  number  = {4},
  year    = {1989},
  pages   = {967--973},
  note    = {\doi{10.1090/S0002-9939-1989-0953740-5}}
}

@article{ForstnericKalaj,
  author  = {Forstneri{\v c}, F. and Kalaj, D.},
  title   = {{Schwarz-Pick} lemma for harmonic maps which are conformal at a point},
  journal = {Anal. PDE},
  volume  = {17},
  number  = {3},
  year    = {2024},
  pages   = {981--1003},
  note    = {\doi{10.2140/apde.2024.17.981}}
}

@article{Heinz,
  author  = {Heinz, E.},
  title   = {On one-to-one harmonic mappings},
  journal = {Pacific J. Math.},
  volume  = {9},
  number  = {1},
  year    = {1959},
  pages   = {101--105},
  note    = {\doi{10.2140/pjm.1959.9.101}}
}

@article{HuangZhaoLi,
  author  = {Huang, Z. and Zhao, D. and Li, H.},
  title   = {A boundary {Schwarz} lemma for pluriharmonic mappings between the unit polydiscs of any dimensions},
  journal = {Filomat},
  volume  = {34},
  number  = {9},
  year    = {2020},
  pages   = {3151--3160},
  note    = {\doi{10.2298/FIL2009151H}}
}

@article{JarnickiPflugProduct,
  author  = {Jarnicki, M. and Pflug, P.},
  title   = {Invariant pseudodistances and pseudometrics - completeness and product property},
  journal = {Ann. Polon. Math.},
  volume  = {55},
  year    = {1991},
  pages   = {169--189},
  note    = {\doi{10.4064/ap-55-1-169-189}}
}

@article{KalajVuorinen,
  author  = {Kalaj, D. and Vuorinen, M.},
  title   = {On harmonic functions and the {Schwarz} lemma},
  journal = {Proc. Amer. Math. Soc.},
  volume  = {140},
  number  = {1},
  year    = {2012},
  pages   = {161--165},
  note    = {\doi{10.1090/S0002-9939-2011-10914-6}}
}

@article{KMS2025,
  author  = {Kne{\v z}evi{\'c}, M. and Mateljevi{\'c}, M. and Svetlik, M.},
  title   = {Some estimates for hyperbolic derivative for {HQC} mappings and applications},
  journal = {Results Math.},
  volume  = {80},
  number  = {8},
  year    = {2025},
  note    = {Art. 229; \doi{10.1007/s00025-025-02546-8}}
}

@misc{KnezevicMateljevicTarget2026,
  author       = {Kne{\v z}evi{\'c}, M. and Mateljevi{\'c}, M.},
  title        = {Target geometry in prescribed-value {Schwarz} lemmas for harmonic maps},
  year         = {2026},
  eprint       = {2609.19609},
  archivePrefix= {arXiv},
  note         = {arXiv:2609.19609}
}

@misc{KnezevicMateljevic2026,
  author = {Kne{\v z}evi{\'c}, M. and Mateljevi{\'c}, M.},
  title  = {An extremal {Schwarz} lemma for harmonic self-mappings of the disk with prescribed value},
  year   = {2026},
  month  = jun,
  note   = {ResearchGate preprint; \doi{10.13140/RG.2.2.10661.46561}}
}

@book{Kobayashi1998,
  author    = {Kobayashi, S.},
  title     = {Hyperbolic Complex Spaces},
  series    = {Grundlehren der Mathematischen Wissenschaften},
  volume    = {318},
  publisher = {Springer},
  address   = {Berlin},
  year      = {1998},
  note      = {\doi{10.1007/978-3-662-03582-5}}
}

@article{Mateljevic2018,
  author  = {Mateljevi{\'c}, M.},
  title   = {{Schwarz} lemma and {Kobayashi} metrics for harmonic and holomorphic functions},
  journal = {J. Math. Anal. Appl.},
  volume  = {464},
  number  = {1},
  year    = {2018},
  pages   = {78--100},
  note    = {\doi{10.1016/j.jmaa.2018.03.069}}
}

@article{MateljevicSvetlik,
  author  = {Mateljevi{\'c}, M. and Svetlik, M.},
  title   = {Hyperbolic metric on the strip and the {Schwarz} lemma for {HQR} mappings},
  journal = {Appl. Anal. Discrete Math.},
  volume  = {14},
  number  = {1},
  year    = {2020},
  pages   = {150--168},
  note    = {\doi{10.2298/AADM200104001M}}
}

@book{Pontryagin,
  author    = {Pontryagin, L. S. and Boltyanskii, V. G. and Gamkrelidze, R. V. and Mishchenko, E. F.},
  title     = {The Mathematical Theory of Optimal Processes},
  publisher = {Interscience},
  address   = {New York},
  year      = {1962}
}

@article{Wegmann,
  author  = {Wegmann, R.},
  title   = {Extremal problems for harmonic mappings from the unit disc to convex regions},
  journal = {J. Comput. Appl. Math.},
  volume  = {46},
  number  = {1-2},
  year    = {1993},
  pages   = {165--181},
  note    = {\doi{10.1016/0377-0427(93)90293-K}}
}

@article{Yau,
  author  = {Yau, S.-T.},
  title   = {A general {Schwarz} lemma for {K\"ahler} manifolds},
  journal = {Amer. J. Math.},
  volume  = {100},
  number  = {1},
  year    = {1978},
  pages   = {197--203},
  note    = {\doi{10.2307/2373880}}
}
\endgroup

\end{document}